\documentclass[10pt]{amsart}
\usepackage{eepic}
\usepackage{epsfig}
\usepackage{graphicx}
\usepackage{eufrak}
\usepackage{mathrsfs}
\usepackage[english]{babel}
\usepackage{color}

\usepackage{palatino, mathpazo}
\usepackage{amscd}      
\usepackage{amssymb}
\usepackage{amsmath}
\usepackage{dsfont}
\usepackage{xypic}      
\LaTeXdiagrams          
\usepackage[all,v2]{xy}
\xyoption{2cell} \UseAllTwocells \xyoption{frame} \CompileMatrices
\usepackage{latexsym}
\usepackage{epsfig}
\usepackage{amsfonts}
\usepackage{enumerate}

\allowdisplaybreaks[3]

\newtheorem{prop}{Proposition}[section]
\newtheorem{lem}[prop]{Lemma}

\newtheorem{thm}[prop]{Theorem}
\newtheorem{rmk}[prop]{Remark}

\newtheorem{defn}[prop]{Definition}

            {\nolinebreak $\Box$ \end{trivlist}}

\newcommand{\noprint}[1]{}

\renewcommand{\tilde}{\widetilde}

\newcommand{\Ext}{\mbox{Ext}}
\newcommand{\spf}{\mbox{spf}}
\newcommand{\Hom}{\mbox{Hom}}

\newcommand{\ext}{{\mbox{\tiny Ext}}}
\newcommand{\an}{{\mbox{\tiny an}}}

\newcommand{\virt}{\mbox{\tiny virt}}

\newcommand{\X}{\mathop{\sf X}\nolimits}
\newcommand{\Y}{\mathop{\sf Y}\nolimits}

\newcommand{\XX}{{\mathfrak X}}

\newcommand{\YY}{{\mathfrak Y}}
\newcommand{\UU}{{\mathfrak U}}

\newcommand{\ZZ}{{\mathfrak Z}}

\newcommand{\FF}{{\mathfrak F}}

\newcommand{\MM}{{\mathfrak M}}

\newcommand{\zz}{{\mathbb Z}}

\newcommand{\kk}{{\mathbb K}}
\newcommand{\T}{{\mathbb T}}

\newcommand{\aaa}{{\mathbb A}}

\newcommand{\nn}{{\mathbb N}}

\renewcommand{\ll}{{\mathbb L}}

\newcommand{\pp}{{\mathbb P}}
\newcommand{\cc}{{\mathbb C}}
\newcommand{\rr}{{\mathbb R}}

\newcommand{\lL}{{\mathcal L}}
\newcommand{\ii}{{\mathcal I}}

\newcommand{\sS}{{\mathcal S}}
\newcommand{\pP}{{\mathcal P}}

\newcommand{\oO}{{\mathcal O}}

\newcommand{\sX}{{\mathcal X}}
\newcommand{\sY}{{\mathcal Y}}
\newcommand{\mM}{{\mathcal M}}

\newcommand{\yY}{{\mathcal Y}}
\newcommand{\zZ}{{\mathcal Z}}

\newcommand{\aA}{{\mathcal A}}
\newcommand{\stft}{\mbox{stft}}
\newcommand{\Coh}{\mbox{Coh}}

\newcommand{\A}{\ensuremath{\mathscr{A}}}
\newcommand{\bC}{{\mathbf C}}
\newcommand{\GBSRig}{\text{GBSRig}}
\newcommand{\BSRig}{\text{BSRig}}
\DeclareMathOperator{\MV}{MV}
\DeclareMathOperator{\DP}{DP}
\DeclareMathOperator{\Def}{Def}
\DeclareMathOperator{\RDef}{RDef}
\DeclareMathOperator{\Var}{Var}
\DeclareMathOperator{\loc}{loc}
\DeclareMathOperator{\val}{val}
\DeclareMathOperator{\id}{id}
\DeclareMathOperator{\Crit}{Crit}
\DeclareMathOperator{\St}{St}
\DeclareMathOperator{\Sch}{Sch}
\DeclareMathOperator{\reg}{reg}
\DeclareMathOperator{\ssc}{ssc}
\DeclareMathOperator{\DT}{DT}

\newcommand{\bX}{{\mathbf X}}
\newcommand{\bt}{{\mathbf t}}

\newcommand{\bs}[1]{\boldsymbol{#1}}
\newcommand{\bmM}{\bs{\mathcal{M}}}
\newcommand{\bmX}{\bs{\mathcal{X}}}

\newcommand{\C}{\ensuremath{\mathscr{C}}}

\newcommand{\Gr}{\mathop{\rm Gr}\nolimits}

\newcommand{\ord}{\mathop{\rm ord}\nolimits}

\newcommand{\red}{\mathop{\rm red}\nolimits}

\newcommand{\spec}{\mathop{\rm Spec}\nolimits}

\numberwithin{equation}{subsection}
\newcommand{\sss}{\vspace{5pt} \subsubsection*{ }\refstepcounter{equation}{{\bfseries(\theequation)}\ }}

\def\Label{\label}

\title[Motivic Joyce-Song formula for the Behrend function identities]{On motivic Joyce-Song formula for the Behrend function identities}
\author{Yunfeng Jiang}
\address{Department of Mathematics\\ University of Kansas\\ 405 Snow Hall 1460 Jayhawk Blvd\\Lawrence KS 66045 USA} 
\email{y.jiang@ku.edu}



\begin{document}
\sloppy \maketitle
\begin{abstract}
We prove the version of Joyce-Song formula for the Behrend function identities in the motivic setting. The main method we use is  the proof of Kontsevich-Soibelman conjecture about the motivic Milnor fibers by Q. T.  Le, who uses the method of motivic integration for formal schemes and Cluckers-Loeser's motivic constructible functions.   In the Appendix the motivic formula can be used to  provide a different proof  that there is an algebra homomorphism of Kontsevich-Soibelman from the motivic  Hall algebra of the abelian category of coherent sheaves on a Calabi-Yau threefold $Y$ to the motivic quantum torus of $Y$.  
\end{abstract}

\maketitle


\section{Introduction}

 \subsection{Background on the Joyce-Song fomula}
\sss Let $Y$ be a smooth Calabi-Yau threefold or a smooth threefold Calabi-Yau Deligne-Mumford stack. The Donaldson-Thomas invariants of $Y$ count stable coherent sheaves on $Y$.  R. Thomas   \cite{Thomas} constructed a perfect obstruction theory $E^\bullet$ in the sense of Li-Tian \cite{LT}, and Behrend-Fantechi \cite{BF} on the moduli space $X$ of stable sheaves over $Y$.   If $X$ is proper, then the virtual dimension of $X$ is zero, and the integral
$\DT_Y=\int_{[X]^{\virt}}1$
is the Donaldson-Thomas invariant of $Y$. In \cite{Behrend} Behrend proves that the moduli scheme $X$ of stable sheaves on $Y$ admits a symmetric obstruction theory which is defined by him in the same paper \cite{Behrend}.  Also Behrend  constructs a canonical integer-valued constructible function 
$$\nu_{X}: X\to\zz$$
on $X$, which we call the {\em Behrend function} of $X$.  More details on the Behrend function is surveyed in \cite{Jiang1}.  If $X$ is proper, then in \cite[Theorem 4.18]{Behrend} Behrend proves that 
$\DT_Y=\int_{[X]^{\virt}}1=\chi(X,\nu_{X})$, 
where $\chi(X,\nu_{X})$ is the weighted Euler characteristic weighted by the Behrend function.  Same result for a proper Deligne-Mumford stack $X$ with a symmetric perfect obstruction theory is conjectured by Behrend in \cite{Behrend}, and  is proved in \cite{Jiang3}.

\sss \label{Joyce-Song:formula} The perfect obstruction theory on the moduli scheme requires that we only can count stable coherent sheaves on $Y$. 
In order to count semi-stable sheaves on the abelian category $\A:=\Coh(Y)$ of coherent sheaves on $Y$,  Joyce-Song in \cite{JS} developed a theory of generalized Donaldson-Thomas invariants.   Let $\MM$ be the moduli stack of coherent sheaves on $\A$, which is an Artin stack  locally of finite type. Then in \cite{JS}, Joyce-Song generalized the definition of the Behrend function to $\MM$: 
$$\nu_{\MM}: \MM\to\zz. $$
We can understand the Behrend function $\nu_{\MM}$ as follows: if there is a finite $1$-morphism
$$f: X\to\MM$$
from a $\kappa$-scheme $X$ to $\MM$, then $f^{*}\nu_{\MM}=(-1)^{n}\nu_{X}$, where $n$ is the relative dimension. 
For any $E_1, E_2\in\Coh(Y)$, Joyce-Song in \cite[\S 5.2]{JS} proves the following formula of the Behrend function identities:
\begin{enumerate}
\item $$\nu_{\MM}(E_1\oplus E_2)=(-1)^{\chi(E_1,E_2)}\nu_{\MM}(E_1)\nu_{\MM}(E_2)\,.$$
Here, $\chi(E_1,E_2)=\sum_i(-1)^i\dim \Ext^i(E_1,E_2)$ is the Euler form.
\item 
$$\int_{F\in \mathbb{P}(\ext^{1}(E_2,E_1))}\nu_{\MM}(F)d\chi-\int_{F\in \mathbb{P}(\ext^{1}(E_1,E_2))}\nu_{\MM}(F)d\chi$$
$$=(\dim(\Ext^{1}(E_2,E_1))-\dim(\Ext^{1}(E_1,E_2)))\nu_{\MM}(E_1\oplus E_2).$$
\end{enumerate}
Here for the integral $\int_{F\in \mathbb{P}(\ext^{1}(E_2,E_1))}\nu_{\MM}(F)d\chi$, we understand it as the weighted Euler characteristic. 
The Formulas (1), (2) are essential to the wall-crossing of Donaldson-Thomas invariants as studied in \cite{JS}, and \cite{Bridgeland11}, since they imply that the morphism from the motivic Hall algebra of $\A$ to the ring of functions of the quantum torus is a Poisson algebra homomorphism. Then the wall-crossing techniques can be applied to get relations between generalized Donaldson-Thomas invariants. 

\sss Let $D^b(\A):=D^b(\Coh(Y))$ be the bounded derived category of coherent sheaves on $Y$. 
An object $E\in D^b(\A)$ is called semi-Schur if $\Ext^{<0}(E,E)=0$. 
It is very interesting to study these formulas for semi-Schur objects in the derived category $D^b(\Coh(Y))$ of coherent sheaves on $Y$.   Note that in \cite{Bussi}
V. Bussi uses the $(-1)$-shifted symplectic structure on the moduli stack $\MM$ of coherent sheaves to prove such Behrend function identities, where her proof relies on the local structure of the moduli stack in \cite{Joyce}. 
In \cite{Jiang2}, we use   Berkovich spaces to prove these formulas.

\subsection{Motivic Joyce-Song formula}

\sss
We follow the proposal of Joyce-Song in \cite{JS} to study the motivic Donaldson-Thomas invariants. 
The moduli stack $\MM$ of objects in $\A$  is an Artin stack,  locally of finite type.  Recall that an  object $E\in \A$ is called   
$semi-Schur$ if it satisfies the condition that 
$\Ext^{<0}(E,E)=0$.  
There is a cyclic dg Lie algebra $R\Hom(E,E)$ corresponding to a semi-Schur object  $E$.  On the cohomology $L_E:=\Ext^*(E,E)$ there is a cyclic $L_\infty$-algebra structure coming from the transfer theorem. 
In \cite{Jiang}, \cite{Jiang2}, we define the Euler characteristic $\chi(E)$ of $E$ by the Euler characteristic of the cyclic 
$L_\infty$-algebra $\Ext^*(E,E)$ or the dg Lie algebra $R\Hom(E,E)$.  Donaldson-Thomas invariants count  stable objects in the derived category and  this Euler characteristic is equal to the pointed Donaldson-Thomas invariant given by the point $E$ in the moduli space.

\sss
If $E$ is  semi-Schur, the cyclic $L_\infty$-algebra $\Ext^*(E,E)$  defines a potential function 
$$f: \Ext^1(E,E)\to\aaa_{\kappa}^1$$
on $\Ext^1(E,E)$, see \cite{Jiang}. In general, $f$ is a formal power series. 

In the case of coherent sheaves,  Joyce-Song prove that $f$ is actually holomorphic, see \cite{JS}. 
For semi-Schur objects, 
Behrend and Getzler, in their unpublished preprint \cite{BG}, proves that $f$ is a  holomorphic function in the complex analytic topology. 
In \cite{Joyce}, Joyce etc use $(-1)$-shifted symplectic structure of \cite{PTVV} on the moduli  space $\MM$ of stable sheaves over smooth Calabi-Yau threefolds to show that the  moduli scheme  locally is given by the critical locus of a regular function $g$.  
The Euler characteristic of the topological Milnor fiber associated with the regular function $g$ gives the pointed Donaldson-Thomas invariant. This regular function may not coincide with the superpotential function $f$ coming from the $L_\infty$-algebra at $E$, but they give the same formal germ moduli scheme $\widehat{\MM}_{E}$ at the point $E$ due to the fact that the germ moduli scheme is the critical locus of the local potential function. 
An argument of this result for coherent sheaves can be found in \cite{FLM},  \cite{FMM}.

\sss \label{analytic:Milnor:fiber}
Let $\kk$ be a non-archimedean complete discretely valued field of characteristic zero.  The ring of integers of $\kk$ is denoted by 
$R$, and the residue field is denoted by $\kappa$.  
Our main example is  $R=\kappa[\![t]\!]$, and the corresponding nonarchimedean field $\kk=\kappa(\!(t)\!)$. 

Associated with the formal potential function $f$, there is a generically smooth {\em special} formal $R$-scheme: 
$$\hat{f}: \XX\to \spf(R),$$
see  \cite{Ber2}, \cite{Nicaise}. 
If $f$ is a regular function, then $(\XX,\hat{f})$ is  the $t$-adic completion of the morphism $f: \Ext^1(E,E)\to\kappa=\spec(\kappa[t])$. 
 The generic fiber 
$\XX_\eta$ is a rigid $\kk$-variety, or a Berkovich space  in sense of \cite{Ber2}.
There exists a specialization map
$$sp: \XX_\eta\to\XX_0$$
from the generic fiber to the reduction $\XX_0$, which is a $\kappa$-variety.  For any 
$y\in\XX_0$, the {\em Analytic Milnor Fiber} $\FF_y(f)$ of $y$ is defined as 
$$\FF_y(f):=sp^{-1}(y).$$
The analytic Milnor fiber $\FF_y(f)$ is an analytic subspace of $\XX_\eta$. 
If we let 
$$\hat{f}_{y}: \XX_y:=\spf(\widehat{\oO}_{\XX,y})\to\spf(R)$$
to be the formal completion of $\XX$ along $y\in\XX_0$, then from \cite{Nicaise} the analytic Milnor fiber 
$\FF_y(f)$ is the generic fiber of the formal scheme $\XX_y$.

\sss
The formal $R$-scheme $(\XX,\hat{f})$ is quasi-excellent in sense of Temkin \cite{Temkin}.  Let 
$$h: \YY\longrightarrow \XX$$
be the resolution of singularities of the formal scheme 
$\XX$. Let $E_i, i\in \ii$ be the set of irreducible components of the exceptional divisors of $h$.  For any 
$I\subset \ii$ let
$$E_{I}:=\bigcap_{i\in I}E_i$$ and 
$$E^{\circ}_{I}:=E_{I}\setminus \bigcup_{j\notin I}E_j.$$
Let $m_{I}=\gcd(m_i)_{i\in I}$, where $m_i$ are the multiplicities of the components $E_i$.  Then there is an Galois cover
$$\widetilde{E}_{I}^{\circ}\to E^{\circ}_{I}$$
with Galois group $\mu_{m_{I}}$.  Hence we get an $\hat{\mu}$-action on $\widetilde{E}_{I}^{\circ}$. 
See \S \ref{Resolution:singularities:formal:scheme} and \cite{Nicaise} for more details on the resolution of singularities.
The following definition is given in  \cite{Jiang}, \cite{Jiang2}. 

\begin{defn}\label{intro:motivic:Milnor:fiber}
The motivic Milnor fiber of the object $E$ is defined as follows:
$$\sS_{0}(E):=\sS_0(\hat{f}):=\sum_{\emptyset\neq I\subset \ii}(1-\mathbb{L})^{|I|-1}[\tilde{E}^{\circ}_{I}\cap h^{-1}(0)].$$
\end{defn}
It is clear that $\sS_{0}(E)\in \mM_{\kappa}^{\hat{\mu}}$.  From \cite{Nicaise}, the motivic volume of the analytic Milnor fiber is given by the motivic Milnor fiber, which  we review in \S \ref{motivic:integration:rigid}. 

Of course, if we have a formal subscheme $\ZZ\subset \XX$, then we define $\sS_{\ZZ}(\hat{f})$ to be the motivic Milnor fiber of 
$\ZZ$: 
$$\sS_{\ZZ}(\hat{f}):=\sum_{\emptyset\neq I\subset \ii}(1-\mathbb{L})^{|I|-1}[\tilde{E}^{\circ}_{I}\cap h^{-1}(\ZZ)].$$

\sss We introduce the following  localized ring of motives:  
$$\mM_{X,\loc}=\mM_{X}[\ll^{-1/2}, (\ll^i-1)^{-1}, i\in\nn_{>0} ]$$ and 
$$\mM^{\hat{\mu}}_{X,\loc}=\mM^{\hat{\mu}}_{X}[\ll^{-1/2}, (\ll^i-1)^{-1}, i\in\nn_{>0} ].$$ 

Let $E_1, E_2, E_1\oplus E_2$ be semi-Schur objects in the derived category of coherent sheaves over $Y$. We prove   the motivic version of Joyce-Song formulas in \cite{Jiang2}.
First we have:
$$\Ext^1(E,E)=\Ext^1(E_1, E_1)\oplus \Ext^1(E_2, E_2)\oplus \Ext^1(E_1, E_2)\oplus \Ext^1(E_2, E_1).$$

\begin{thm}\label{main:thm}
For polynomial potential functions,  in $\mM^{\hat{\mu}}_{\kappa,\loc}$ have the following two formula
\begin{enumerate}

\item $$(1-\mathcal{S}_{((0,0))}(E_1\oplus E_2))=(1-\sS_{0}(E_1))\cdot (1-\sS_{0}(E_2))\,.$$
\item 
\begin{align*}
&\int_{F\in\mathbb{P}(\ext^{1}(E_2,E_1))}(1-\sS_{0}(F))-\int_{F\in\mathbb{P}(\ext^{1}(E_1,E_2))}(1-\sS_{0}(F))\\
&=([\pp^{\dim\ext^{1}(E_2,E_1)}]-[\pp^{\dim\ext^{1}(E_1,E_2)}])\left(1-\sS_{f|_{\ext^1(E_1, E_1)\oplus \ext^1(E_2, E_2)},0}\right).
\end{align*}
\end{enumerate}
where $\int_{\MM_0}(-): \mM^{\hat{\mu}}_{\MM_0}\to \mM^{\hat{\mu}}_{\kappa}$ is the pushforward of motives. 
\end{thm}

\sss We give an explanation about the formulas.  For any semi-Schur object $E\in D^b(\A)$, $\sS_{0}(E)$ is the motivic Milnor fiber of $E$, and $(1-\sS_0(E))$ is the analogue of motivic vanishing cycle.  Let $E:=E_1\oplus E_2$.  Let
$$\phi: \widetilde{\XX}\to \XX:=\widehat{\Ext^{1}(E,E)}\to\spf(R)$$
be the formal blow-up of $\XX$ along
the completion $\YY\subset \XX$, where 
$\YY=\widehat{V}$ is the formal completion of $V$,  and $V:=\Ext^1(E_1,E_1)\oplus \Ext^1(E_2,E_2)\oplus 0\oplus\Ext^1(E_2,E_1)\subset \Ext^1(E,E)$.
We denote by $\ZZ:=\widehat{\Ext^1(E_1, E_2)}\subset \XX$. 
Let $\pp(\ZZ):=\widehat{\pp(\Ext^1(E_1,E_2))}\subset \widetilde{\XX}$ be the closed formal  subscheme of $\widetilde{\XX}$. 
The corresponding reduction scheme is denoted by 
$\pp(\ZZ)_0=\pp(\Ext^1(E_1,E_2))$.
Since the motivic vanishing cycle is constructible,   then
the integration 
$$\int_{F\in\mathbb{P}(\ext^{1}(E_2,E_1))}(1-\sS_{0}(F))$$
can be understood as the motivic cycle
$\sS_{\pp(\ZZ)_0}(\widetilde{\hat{f}})$, where 
$\widetilde{\hat{f}}=\phi\circ \hat{f}: \widetilde{\XX}\to\spf(R)$ is the formal $R$-scheme $\widetilde{\XX}$ of the composition 
$\phi\circ \hat{f}$.

\begin{rmk}
The Euler characteristic of the motivic Milnor fiber $\sS_{0}(E)$ is, plus the correct sign, the value of the Behrend function 
$\nu_{\MM}$ on $E\in\MM$.  Hence taking the Euler characteristic on the formulas (1), (2) in Theorem  \ref{main:thm},  when putting the right signs,  we get the Joyce-Song formula (1), (2) in \S (\ref{Joyce-Song:formula}).
\end{rmk}

\sss Our proof of Theorem \ref{main:thm} is motivated by  Le's  study of Kontsevich-Soibelman Conjecture in \cite{Le}, \cite{Le2}. 
The formula (1) can be taken as a refinement of Le's result, and the second formula (2) is just from the conjecture of Kontsevich-Soibleman directly.   Hence our new result here is only  the formula (1). 

\begin{rmk}
The motivic Joyce-Song formula should be true for any formal potential function.  
In the former version of the paper we use the argument of Le in \cite{Le2}. It seems that the proof there has some issues, and to be safe we only prove the result for polynomial potential functions. 
\end{rmk}

\subsection{The algebra homomorphism}

\sss In the appendix we use  the motivic Joyce-Song formulas to prove an algebra homomorphism from the motivic Hall algebra $H(\A)$ to the motivic quantum torus 
$\overline{\mM}^{\hat{\mu}}_{\kappa,\loc}[\Gamma]$, see Theorem \ref{main:homomorphism}.  In \cite{KS}, Kontsevich-Soibelman shows that there is an algebra homomorphism from the motivic Hall algebra $H(\A)$ to the motivic quantum torus $\overline{\mM}^{\hat{\mu}}_{\kappa,\loc}[\Gamma]$ based on a conjecture of motivic nearby cycles, which is proved by Le \cite{Le}.   Le's result may also implies the algebra homomorphism in Theorem \ref{main:homomorphism}.  Since we use Joyce's $d$-critical locus, we give a new proof for this homomorphism. 

In an old version of this paper we had tried to prove that there is a Poisson algebra homomorphism 
from the motivic Hall algebra $H(\A)$ to the motivic quantum torus 
$\overline{\mM}^{\hat{\mu}}_{\kappa,\loc}$ which generalizes the  
 Lie algebra homomorphism in \cite[Theorem 5.14]{JS}, and the Poisson algebra homomorphism in \cite[Theorem 5.2]{Bridgeland10} to the motivic level.  But it seems it is difficult to have a correct definition of 
 Poisson bracket on the motivic quantum torus 
$\overline{\mM}^{\hat{\mu}}_{\kappa,\loc}$.  We only deal with the algebra homomorphism, which is already proved by Kontsevich-Soibelman in \cite{KS}.

Here the ring $\overline{\mM}^{\hat{\mu}}_{\kappa,\loc}[\Gamma]$ is roughly defined as follows.   The  ring 
$\mM^{\hat{\mu}}_{\kappa,\loc}[\Gamma]$ is a formal power series ring over $\mM^{\hat{\mu}}_{\kappa,\loc}$ generated by symbols $x^\alpha$ for 
$\alpha\in\Gamma$, where $\Gamma$ is the effective classes of the numerical K-group of $Y$.  
The ring $\overline{\mM}^{\hat{\mu}}_{\kappa,\loc}[\Gamma]$ is the quotient of the ring $\mM^{\hat{\mu}}_{\kappa,\loc}[\Gamma]$ modulo the  relations
$$\Upsilon(Q_1\otimes Q_2)-\Upsilon(Q_1)\odot\Upsilon(Q_2)$$
for quadratic forms $Q_1, Q_2$ and $\Upsilon(Q_i)$ are the motive of the quadratic forms for $i=1,2$. 
This is related to the triangle property of the orientation data in \cite{KS} and  have applications to the wall crossing of motivic Donaldson-Thomas invariants.

\sss 
The motivic Hall algebra $H(\A)$ is a $K(\Var_{\MM})[\ll^{-1}]$-module. 
We define a submodule of $H(\A)$ by the elements $[X\to\MM]$ such that 
$X$ is an algebraic $d$-critical locus in the sense of \cite{Joyce}.  We call it the $d$-critical elements of 
$H(\A)$ and denote it by $H_{d-\Crit}(\A)$.  Then let 
$$H_{\ssc, d-\Crit}(\A)=H_{d-\Crit}(\A)/(\ll-1)H_{d-\Crit}(\A).$$ 
We define the integration map 
$$I: H_{\ssc, d-\Crit}(\A)\to \overline{\mM}_{\kappa}^{\hat{\mu}}[\Gamma]$$
by taking the global motivic sheaf $\sS_{X}^{\phi}$ for the algebraic $d$-critical locus $X$. 
By \cite{BBJ}, \cite{BJM}, if the algebraic $d$-critical locus $(X,s)$ has an orientation, which is a root line bundle $K_{X,s}^{1/2}$
for the canonical line bundle $K_{X,s}$, then there exists a global motivic sheaf $\sS_{X}^{\phi}\in \overline{\mM}_{X}^{\hat{\mu}}$, where $\overline{\mM}_{X}^{\hat{\mu}}$ is defined similarly to 
$\overline{\mM}_{\kappa}^{\hat{\mu}}$ by considering the motives of quadratic forms over $X$.   The sheaf  $\sS_{X}^{\phi}$, when restricted to the local critical chart of $X$, is the perverse sheaf of vanishing cycles times the motive of a quadratic form over $X$. 
In this paper we always assume that there exists  an orientation. 
Please see \S \ref{Application:homo} for more details.
The algebra $H_{\ssc, d-\Crit}(\A)$ is called the semi-classical part of the Hall algebra.We prove that the integration map $I$ is an algebra homomorphism.   

\sss We give an explanation on how the formulas in Theorem \ref{main:thm} can be used in the proof of the  algebra homomorphism for the integration map in \S \ref{proof:homomorphism}.  Let $(U,g)$ be a critical chart around  a point $E$ of the algebraic $d$-critical locus 
$X\subset \MM$. The global motive $\sS_{X}^{\phi}\in \overline{\mM}_{X}^{\hat{\mu}}$ is given by the sheaf of vanishing cycles 
$\sS_{U, \hat{g}}^{\phi}=\mathds{1}-\sS_{U,g}\in \overline{\mM}_{X}^{\hat{\mu}}$. 
Let 
$$\hat{g}: \UU\to\spf(R)$$
be the formal completion of $g$ along the origin. 
The nearby cycle $\sS_{U,g}$ can be given by the formally setting nearby cycle $\sS_{\UU,\hat{g}}=\sS_{0}(\hat{g})$, which is defined in 
Definition \ref{intro:motivic:Milnor:fiber}.  The motivic  Milnor fiber $\sS_{0}(\hat{g})$ is isomorphic to the  Milnor fiber 
$\sS_{0}(\hat{f}_{E})$, where $f_{E}: \Ext^1(E,E)\to\cc$ is the superpotential function given by the cyclic $L_\infty$-algebra on $\Ext^*(E,E)$. 
Actually these two formal schemes $(\XX_E, \hat{f}_{E})$ and $(\UU, \hat{g})$ are isomorphic, since they represent the same formal germ moduli scheme $\widehat{\MM}_{E}$.  From  \cite[Theorem 8.8]{Nicaise} and \cite[Theorem 9.4]{Nicaise}, the analytic Milnor fibers $\FF_{0}(\hat{f}_{E})$ and $\FF_0(\hat{g})$ are isomorphic over $\kk$ and their corresponding motivic Milnor fibers $\sS_{0}(\hat{g})$ and 
$\sS_{0}(\hat{f}_{E})$ are isomorphic as motives.  Then the formulas in  Theorem \ref{main:thm} implies that the integration map $I$ is an algebra homomorphism, see \S \ref{proof:homomorphism}.

\subsection{Outline}

\sss The outline of the paper is as follows.  The materials about motivic integration are reviewed in \S \ref{motivic:integration}, where in \S \ref{Grothendieck:ring} we review the Grothendieck ring of varieties, and 
in \S \ref{motivic:integration:rigid} we briefly talk about the motivic integration of rigid varieties from formal scheme models following \cite{Nicaise}.  In \S \ref{proof:conjecture} we prove Theorem \ref{main:thm}.  
In \S \ref{CL:motivic:constructible} we review the techniques of motivic constructible functions in \cite{CL};  in \S \ref{proof:formula1} we prove Formula (1) in Theorem \ref{main:thm}; and in \S \ref{proof:formula2} we prove Formula (2) in Theorem \ref{main:thm}. 
Combining sections \S \ref{proof:formula1} and \S \ref{proof:formula2}, Theorem \ref{main:thm} is proved. 
Section \S  \ref{Application:homo} serves as the proof of the algebra homomorphism from the motivic Hall algebra to the motivic quantum torus, where in \S \ref{sec:motivic:Hall} we introduce the motivic Hall algebra 
$H(\A)$ for the abelian category $\A$; in \S \ref{sec:algebraic:d:critical:locus} we briefly review the notion of algebraic $d$-critical locus of Joyce in \cite{Joyce};  in \S \ref{motivic:integration:map} we define the integration map; and in \S \ref{proof:homomorphism} we prove that  the integration map is an algebra homomorphism.

\subsection*{Convention}

Throughout the paper we work over an  algebraically closed field $\kappa$ so that the nonarchimedean field is 
$\kappa(\!(t)\!)$ and its ring of integers is $R=\kappa[\![t]\!]$.  For the applications in \S \ref{Application:homo}, we consider the schemes and stacks over $\kappa=\cc$, the field of complex numbers. 

For a Berkovich analytic space $\FF$,  we use  $\chi(\FF)$ to represent the Euler characteristics the \'etale cohomology of $\FF$. 
We use $\ll$ to represent the Lefschetz motive $[\aaa^1_{\kappa}]$. 
\subsection*{Acknowledgments}

The author would like to thank Kai Behrend, Johannes Nicaise, Sam Payne and Andrew MacPherson  for valuable discussions on Berkovich spaces,  especially Johannes Nicaise for answering the questions about the motivic integration of formal schemes in \cite{Nicaise}.  Many thanks to Andrea Ricolfi and J$\o$rgen Rennemo for his interests and valuable discussions. 
He also thanks  Professor Tom Bridgeland for the discussion of the integration map in his proof of the DT/PT-correspondence, 
Professor Dominic Joyce for the  discussion of the Joyce-Song formula for the Behrend function identities, and  Professors Tom Coates, Richard Thomas and Alessio Corti for support at Imperial College London.
This work is partially supported by  Simons Foundation Collaboration Grant 311837, and NSF Grant DMS-1600997.

\section{The motivic Milnor fibre and the motivic volumes.}\label{motivic:integration}

\subsection{Grothendieck group of varieties.}\label{Grothendieck:ring}

\sss In this section we briefly review the Grothendieck group of varieties. 
Let $S$ be an algebraic variety over $\kappa$. Let $\Var_{S}$ be the category of 
$S$-varieties.

\sss Let $K_0(\Var_{S})$ be the Grothendieck group of $S$-varieties.  By definition $K_0(\Var_{S})$ 
is an abelian group with generators given by all the varieties $[X]$'s, where $X\rightarrow S$ are $S$-varieties,  and the relations are $[X]=[Y]$, if $X$ is isomorphic to $Y$, and 
$[X]=[Y]+[X\setminus Y]$ if $Y$ is a Zariski closed subvariety of $X$.
Let $[X],  [Y]\in K_0(\Var_{S})$,  and define $[X][Y]=[X\times_{S} Y]$.  Then 
we have a product on $K_0(\Var_{S})$. 
Let $\mathbb{L}$ represent the class of $[\mathbb{A}_{\kappa}^{1}\times S]$.
Let $\mathcal{M}_{S}=K_0(\Var_{S})[\mathbb{L}^{-1}]$
be the ring by inverting the class $\mathbb{L}$ in the ring $K_0(\Var_{S})$.

If $S$ is a point $\spec (\kappa)$, we write $K_0(\Var_{\kappa})$ for the Grothendieck ring of $\kappa$-varieties.
One can take the map $\Var_{\kappa}\longrightarrow K_0(\Var_{\kappa})$ to be the universal Euler characteristic.
After inverting the class $\mathbb{L}=[\mathbb{A}_{\kappa}^{1}]$, we get the ring $\mathcal{M}_{\kappa}$.

\sss We introduce the equivariant Grothendieck group defined in \cite{DeLo1}.
Let $\mu_n$ be the cyclic group of order $n$, which can be taken as the algebraic variety
$\spec (\kappa[x]/(x^n-1))$. Let $\mu_{md}\longrightarrow \mu_{n}$ be the map $x\mapsto x^{d}$. Then 
all the groups $\mu_{n}$ form a projective system. Let 
$$\underleftarrow{lim}_{n}\mu_{n}$$
be the direct limit.

Suppose that $X$ is a $S$-variety. The action $\mu_{n}\times X\longrightarrow X$ is called a $good$ 
action if  each orbit is contained in an affine subvariety of $X$.  A good $\hat{\mu}$-action on $X$ is an action of $\hat{\mu}$ which factors through a good $\mu_n$-action for some $n$.

The $equivariant ~Grothendieck~ group$ $K^{\hat{\mu}}_0(\Var_{S})$ is defined as follows:
The generators are $S$-varieties $[X]$ with a good $\hat{\mu}$-action; and the relations are:
$[X,\hat{\mu}]=[Y,\hat{\mu}]$ if $X$ is isomorphic to $Y$ as $\hat{\mu}$-equivariant $S$-varieties,  
and $[X,\hat{\mu}]=[Y,\hat{\mu}]+[X\setminus Y, \hat{\mu}]$ if $Y$ is a Zariski closed subvariety
of $X$ with the $\hat{\mu}$-action induced from that on $X$,  if $V$ is an affine variety with a good 
$\hat{\mu}$-action, then $[X\times V,\hat{\mu}]=[X\times \mathbb{A}_{\kappa}^{n},\hat{\mu}]$.  The group 
$K^{\hat{\mu}}_0(\Var_{S})$ has a ring structure if we define the product as the fibre product with the good $\hat{\mu}$-action.  Still we let $\mathbb{L}$  represent the class $[S\times \mathbb{A}_{\kappa}^{1},\hat{\mu}]$ and let $\mathcal{M}_{S}^{\hat{\mu}}=K^{\hat{\mu}}_0(\Var_{S})[\mathbb{L}^{-1}]$ be the ring obtained from $K^{\hat{\mu}}_0(\Var_{S})$ by inverting the class $\mathbb{L}$.

If $S=\spec(\kappa)$, then we write $K^{\hat{\mu}}_0(\Var_{S})$ as $K^{\hat{\mu}}_0(\Var_{\kappa})$, and $\mathcal{M}_{S}^{\hat{\mu}}$ as $\mathcal{M}_{\kappa}^{\hat{\mu}}$.  Let  $s\in S$ be a geometric point. Then we have natural maps $K^{\hat{\mu}}_0(\Var_{S})\longrightarrow K^{\hat{\mu}}_0(\Var_{\kappa})$ and $\mathcal{M}_{S}^{\hat{\mu}}\longrightarrow \mathcal{M}_{\kappa}^{\hat{\mu}}$ given by the correspondence
$[X,\hat{\mu}]\mapsto [X_s,\hat{\mu}]$.

\sss 
Let $S$ be a scheme.  Following \cite{BJM}, we need to define a new product $\odot$ on $\mM_{S}^{\hat{\mu}}$.
The following definition is due to \cite[Definition 2.3]{BJM}. 
\begin{defn}
Let $[X, \widehat{\sigma}], [Y,\widehat{\tau}]$ be two elements in $K_0^{\hat{\mu}}(\Var_{S})$ or $\mM_{S}^{\hat{\mu}}$. 
Then there exists $n\geq 1$ such that the $\hat{\mu}$-actions $\widehat{\sigma}, \widehat{\tau}$ on $X,Y$ factor through $\mu_n$-actions
$\sigma_n, \tau_n$.  Define $J_n$ to be the Fermat curve
$$J_n=\{(t,u)\in (\aaa^1\setminus \{0\})^2: t^n+u^n=1\}.$$
Let $\mu_n\times\mu_n$ act on $J_n\times(X\times_{S}Y)$ by
$$(\alpha,\alpha^\prime)\cdot ((t,u),(v,w))=((\alpha\cdot t, \alpha^\prime\cdot u), (\sigma_n(\alpha)(v), \tau_n(\alpha^\prime)(w))).$$
Write $J_n(X,Y)=(J_n\times (X\times_{S}Y))/(\mu_n\times\mu_n)$ for the quotient $\kappa$-scheme, and 
define a $\mu_n$-action $v_n$ on $J_n(X,Y)$ by
$$v_n(\alpha)((t,u), v,w)(\mu_n\times\mu_n)=((\alpha\cdot t, \alpha\cdot u),v,w)(\mu_n\times\mu_n).$$
Let $\hat{v}$ be the induced good $\hat{\mu}$-action on $J_n(X,Y)$, and set
$$[X, \widehat{\sigma}]\odot [Y,\widehat{\tau}]=(\ll-1)\cdot [(X\times_{S}Y/\mu_n, \hat{\iota})]-[J_n(X,Y),\hat{v}]$$
in  $K_0^{\hat{\mu}}(\Var_{S})$ or $\mM_{S}^{\hat{\mu}}$. This defines a commutative, associative product on  $K_0^{\hat{\mu}}(\Var_{S})$ or $\mM_{S}^{\hat{\mu}}$.
\end{defn}

Consider the Lefschetz motive $\ll=[\aaa^1_{\kappa}]$. As in \cite{BJM}, we define 
$\ll^{\frac{1}{2}}$ in $K_0^{\hat{\mu}}(\Var_{S})$ or $\mM_{S}^{\hat{\mu}}$ by:
$$\ll^{\frac{1}{2}}=[S,\hat{\iota}]-[S\times\mu_2,\hat{\rho}],$$
where $[S,\hat{\iota}]$ with trivial $\hat{\mu}$-action $\hat{\iota}$ is the identity in  $K_0^{\hat{\mu}}(\Var_{S})$ or $\mM_{S}^{\hat{\mu}}$,
and $S\times\mu_2$ is the two copies of $S$ with the nontrivial $\hat{\mu}$-action $\hat{\rho}$ induced by the left action of $\mu_2$ on itself, exchanging the two copies of $S$.  Then $\ll^{\frac{1}{2}}\odot\ll^{\frac{1}{2}}=\ll$.

\subsection{Motivic integration on rigid varieties}\label{motivic:integration:rigid}

\sss Let $\kk$ be a non-archimedean complete discretely valued field of characteristic zero.  The ring of integers of $\kk$ is denoted by 
$R$, and the residue field is denoted by $\kappa$.  For instance, $R=\kappa[\![t]\!]$ and $\kk=\kappa(\!(t)\!)$ is the fractional field of $R$.  We fix a unformizing paremeter $\pi$ in $R$ throughout the paper, and in the case that $R=\kappa[\![t]\!]$, $\pi=t$. 

\sss Let $\XX\to \spf(R)$ be a  separated generically smooth  formal scheme over $R$ of topologically of finite type.  We call such types of  $R$-formal schemes $stft$ $R$-formal schemes.  A $stft$ $R$-formal scheme $\XX$ is obtained by gluing finite open covers by affine $stft$ formal $R$-schemes.  Each affine  $stft$ formal $R$-scheme is of the form 
$$\spf(A)\to \spf(R)$$
for a topologically of finite type $R$-algebra $A$, which is isomorphic to an algebra of the form 
$R\{x_1,\cdots,x_m\}/I$ for some integer $m>0$ and some ideal $I$, where $R\{x_1,\cdots,x_m\}$ is the algebra of converging power series over $R$.  

The $special~ fiber$ $\XX_0$ for an affine  $stft$ formal $R$-scheme $\XX=\spf(A)$ is the $\kappa$-scheme 
$\XX_0=\spec(A_0)$, where $A_0$ is the $\kappa$-algebra $A/J$ with $J$ the largest ideal of definition.  In general the affine covers of $\spec(A_0)$ glue to give the $\kappa$-scheme $\XX_0$ for any $stft$ formal $R$-scheme $\XX$.

\sss Let $\XX$ be a generically smooth $stft$ formal $R$-scheme. The generic fiber $\XX_\eta$ is rigid $\kk$-variety. The construction is obtained by gluing the constructions on affine charts.  If $\XX=\spf(A)$ is affine, recall that from \cite[\S 4.8, \S 4.9]{Nicaise-survey}, there is a specialization map
$$sp: |\XX_\eta|\to |\XX|=|\XX_0|$$
such that if $\UU$ is any open formal subscheme of $\XX$, then $sp^{-1}(\UU)$ is an admissible open in $\XX_\eta$.  Thus the generic fibers  $\UU_i$ of an affine open covers of a $stft$ formal $R$-scheme $\XX$ can be glued along the generic fibers of the intersections
$\UU_i\cap \UU_j$ to obtain a rigid $\kk$-variety $\XX_\eta$. The specialization maps glue to give a continuous map
\begin{equation}
sp: |\XX_\eta|\to |\XX|=|\XX_0|.
\end{equation} 

In the language of Berkovich analytic spaces, the analytification of $\XX_\eta$ is a Berkovich analytic space over the nonarchimedean field $\kk$ in sense of Berkovich \cite{Ber}.  We still denote it by $\XX_\eta$.  The construction is also obtained by gluing the constructions on affine charts.  For $\XX=\spf(A)$ is affine, let $\aA=A\otimes_{R}\kk$ then $\XX_\eta=\mM(\aA)$ is the  spectrum of the affinoid $\kk$-algebra
$\aA$, which consists of all bounded multiplicative semi-norms 
$x: \aA\to \rr_{+}$.  The affine Berkovich spaces $\mM(\aA)$ glue to give  us the Berkovich analytic space $\XX_\eta$.

Finally we recall the following result in \cite{Nicaise-survey}.  The construction of the generic fiber is functorial.  A morphism of $stft$ formal 
$R$-schemes $h: \YY\to \XX$ induces a morphism of rigid $\kk$-varieties 
$h_\eta: \YY_\eta\to\XX_\eta$ and the square
\begin{equation}\label{commutative-specialization-map}
\xymatrix{
\YY_\eta\ar[r]^{h_\eta}\ar[d]_{sp}& \XX_\eta\ar[d]^{sp}\\
\YY\ar[r]^{h}&\XX
}
\end{equation}
commutes.  Thus there is a functor
$$(\cdot)_{\eta}:  (\stft-For/R)\to (sqc-Rig/\kk): \XX\mapsto \XX_\eta$$
from the category of $\stft$ formal $R$-schemes to the category of separated, quasi-compact rigid $\kk$-varieties. 

\sss  Let $\XX$ be a generically smooth $\stft$ formal $R$-scheme. We follow the construction of Nicaise-Sebag, Nicaise in \cite{NS}, \cite{Nicaise} for the definition of the motivic integration of a gauge form $\omega$ on $\XX_\eta$, which takes values in $\mM_{\XX_0}$. 

We briefly recall the method to define the motivic integration $\int_{\XX}|\omega|$. 
First we have
$$\XX=\varinjlim_{m}\sX_m,$$
where $\sX_m=(\XX, \oO_{\XX}\otimes_{R}R_m)$ and $R_m=R/(\pi)^{m+1}$. In Greenberg \cite{Greenberg}, the functor
$$\yY\mapsto \Hom_{R_m}(\yY\times_{\kappa}R_m, \sX_m)$$
from the category of $\kappa$-schemes to the category of sets is presented by a $\kappa$-scheme 
$$\Gr_{m}(\sX_m)$$
of finite type such that
$$\Gr_m(\sX_m)(A)=\sX_m(A\otimes_{\kappa}R_m)$$ for any $\kappa$-algebra $A$.  The projective limit 
$\varinjlim_{m}\sX_m$ is denoted by $\Gr(\XX)$.  The functor $\Gr$ respects open and closed immersions and fiber products, 
and sends affine topologically of finite type formal $R$-schemes to affine $\kappa$-schemes. The motivic integration of a gauge form 
$\omega$ is defined by using the stable cylindrical subsets of $\Gr(\XX)$, introduced by Loeser-Sebag in \cite{LS}, and Nicaise-Sebag in \cite{NS-curve}. 

Let $\bC_{0,\XX}$ be the set of stable cylindrical subsets of $\Gr(\XX)$ of some level.  If $A\subset \bC_{0,\XX}$ is a cylinder, and we have a function
$$\alpha: A\to \zz\cup\{\infty\}$$
such that $\alpha^{-1}(m)\subset \bC_{0,\XX}$. Then 
$$\int_{A}[\aaa_{\XX_0}^{1}]^{-\alpha}d\widetilde{\mu}:=\sum_{m\in\zz}\widetilde{\mu}(\alpha^{-1}(m))\cdot [\aaa^{1}_{\XX_0}]^{-m},$$
where $$\widetilde{\mu}: \bC_{0,\XX}\to \mM_{\XX_0}$$
is the unique additive morphism defined in \cite[Proposition 5.1]{Le} by
$$\widetilde{\mu}(A)=[\pi_m(A)]\cdot [\aaa^{1}_{\XX_0}]^{-(m+1)d}$$
for $A$ a stable cylinder of level $m$,  $d$ is the relative dimension of $\XX$, and $\pi_m: \Gr(\XX)\to \Gr(\sX_m)$
is the canonical projection. 

Let $\omega$ be a gauge form on $\XX_\eta$, in \cite{LS}, the authors constructed an integer-valued function
$$\ord_{\pi, \XX}(\omega)$$
on $\Gr(\XX)$ that takes the role of $\alpha$ before.  The motivic integration $\int_{\XX}|\omega|$ is defined to be
\begin{equation}\label{defn:motivic:integration}
\int_{\XX}|\omega|:=\int_{\Gr(\XX)}[\aaa_{\XX_0}^{1}]^{-\ord_{\pi,\XX}(\omega)}d\widetilde{\mu}\in \mM_{\XX_0}.
\end{equation}
From \cite{LS}, \cite{NS}, the forgetful map
$$\int: \mM_{\XX_0}\to \mM_{\kappa}$$ defined by
$$\int_{\XX}|\omega|\mapsto \int_{\XX_\eta}|\omega|:=\int \int_{\XX}|\omega|$$
only depends on $\XX_\eta$, not on $\XX$.

\sss In \cite{Nicaise} Nicaise generalizes the motivic integration construction to generically smooth special formal $R$-schemes.  A special formal 
$R$-scheme $\XX$ is a separated Noetherian adic formal scheme endowed with a structural morphism $\XX\to \spf(R)$, such that $\XX$ is a finite union of open formal subschemes which are formal spectra of special $R$-algebras.   From Berkovich \cite{Ber2}, a topological $R$-algebra $A$ is special, iff $A$ is topologically $R$-isomorphic to a quotient of the special $R$-algebra
$$R\{T_1,\cdots, T_m\}[\![S_1,\cdots,S_n]\!]=R[\![S_1,\cdots,S_n]\!]\{T_1,\cdots, T_m\}.$$

The Noetherian adic formal scheme $\XX$ has the largest ideal of definition $J$. The closed subscheme of $\XX$ defined by $J$ is denoted by 
$\XX_0$, which is a reduced Noetherian $\kappa$-scheme.

\sss We briefly review the motivic integration of Nicaise in \cite{Nicaise}. 

\begin{defn}
Let $\XX$ be a special formal $R$-scheme. By a \textbf{N\'eron smoothening} we mean a morphism of special formal 
$R$-schemes $\YY\to\XX$, such that $\YY$ is adic smooth over $R$ and $\YY_\eta\to\XX_\eta$ is an open embedding satisfying 
$\YY_\eta(\widetilde{K})=\XX_\eta(\widetilde{\kk})$ for any finite unramified extension $\widetilde{\kk}$ of $\kk$. 
\end{defn}

In \cite[\S 2]{Nicaise}, Nicaise proves that a N\'eron smoothening of $\XX$ exists and is given by the dilatation of $\XX$.  Then $\YY$ is a 
$\stft$ formal $R$-scheme. 

\begin{defn}
Let $\XX$ be a generically smooth special formal $R$-scheme.  We define
$$\int_{\XX}|\omega|:=\int_{\YY}|\omega|$$ and 
$$\int_{\XX_\eta}|\omega|:=\int_{\YY_\eta}|\omega|$$
for a gauge form $\omega$ on $\XX_\eta$. 
\end{defn}

\sss We recall the motivic volume of $\XX_\eta$ in \cite{Nicaise}.  
For $m\geq 1$, let 
$\kk(m):=\kk[T]/(T^m-\pi)$ be a totally ramified extension of degree $m$ of $\kk$, and 
$R(m):=R[T]/(T^m-\pi)$ the normalization of $R$ in $\kk(m)$.  If $\XX$ is a formal $R$-scheme, we define 
$$\XX(m):=\XX\times_{R}R(m)$$ and
$$\XX_\eta(m):=\XX_\eta\times_{\kk}\kk(m).$$
If $\omega$ is a gauge form on $\XX_\eta$, we denote by $\omega(m)$ the pullback of $\omega$
via the natural morphism $\XX_\eta(m)\to \XX_\eta$. 

\begin{defn}\label{defn_Poincare_series}
Let $\XX$ be a generically smooth special formal $R$-scheme. Let 
$\omega$ be a gauge form on $\XX_\eta$.  Then the volume Poincar\'e series of $(\XX,\omega)$ is defined to be
$$S(\XX,\omega;T):=\sum_{d>0}\left(\int_{\XX(d)}|\omega(d)|\right)T^d\in \mM_{\XX_0}[\![T]\!].$$
\end{defn}

\sss \label{Resolution:singularities:formal:scheme}
\begin{defn}
Let $\XX$ be a generically smooth flat $R$-formal scheme.  A resolution of singularities of $\XX$ is a proper morphism 
$h: \YY\to \XX$ of flat special formal $R$-schemes such that  
$h$ induces an isomorphism on generic fibers, and such that $\YY$ is regular (meaning the local ring at points is regular), with a special fiber 
a strict normal crossing divisor $\YY_s$.  We say that the resolution $h$ is tame if $\YY_s$ is a tame normal crossing divisor. 
\end{defn}

By Temkin's resolution of singularities for quasi-excellent schemes of characteristic zero in  \cite{Temkin},   any affine generically smooth flat special formal $R$-scheme $\XX=\spf(A)$ admits a resolution of singularities by means of admissible blow-ups.   

In general for any generically smooth $R$-formal scheme $\XX$, suppose that there is a resolution of singularities 
\begin{equation}
h:  \YY\longrightarrow \XX
\end{equation}

Let $E_i$, $i\in \ii$, be the set of irreducible components of the exceptional divisors of the resolution. 
For $I\subset \ii$,  we set 
$$E_{I}:=\bigcap_{i\in I}E_{i}$$
and 
$$E_{I}^{\circ}:=E_{I}\setminus \bigcup_{j\notin I}E_j.$$
Let $m_{i}$ be the multiplicity of the component $E_i$, which means that 
the special fiber of the resolution is 
$$\sum_{i\in \ii}m_iE_i.$$
Let $m_{I}=\gcd(m_i)_{i\in I}$. Let $U$ be an affine Zariski open subset of $\YY$, such that, 
on $U$, $f\circ h=uv^{m_{I}}$, with $u$ a unit in $U$ and $v$ a morphism from 
$U$ to $\mathbb{A}_{\cc}^{1}$. The restriction of $E_{I}^{\circ}\cap U$, which we denote by
$\tilde{E}_{I}^{\circ}\cap U$, is defined by
$$\lbrace{(z,y)\in \mathbb{A}_{\cc}^{1}\times (E_{I}^{\circ}\cap U)| z^{m_{I}}=u^{-1}\rbrace}.$$
The $E_{I}^{\circ}$ can be covered by the open subsets $U$ of $Y$.  We can glue together all such 
constructions and get the Galois cover
$$\tilde{E}_{I}^{\circ}\longrightarrow E_{I}^{\circ}$$
with Galois group $\mu_{m_{I}}$.
Remember that $\hat{\mu}=\underleftarrow{lim} \mu_{n}$ is the direct limit of the groups
$\mu_{n}$. Then there is a natural $\hat{\mu}$ action on $\tilde{E}_{I}^{\circ}$.
Thus we get 
$[\tilde{E}_{I}^{\circ}]\in \mathcal{M}_{X_0}^{\hat{\mu}}$.

\sss\label{motivic:result:formal:scheme} Using resolution of singularities, in \cite[Theorem 7.12]{Nicaise}, Nicaise proves the following result:

\begin{thm}\label{motivic:integration:formula:omegam}
Let $\XX$ be a generically smooth special formal $R$-scheme of pure relative dimension $d$. 
Then we have a structural morphism $f: \XX\to \spf(R)$. 
Suppose that $\XX$ has a resolution of singularities $\XX^\prime\to\XX$ with special fiber 
$\XX^\prime_s=\sum_{i\in I}N_i E_i$. 

Let $\omega$ be a $\XX$-bounded gauge form on $\XX_\eta$, where the definition of bounded gauge form is given by Nicaise in \cite[Definition 2.11]{Nicaise}. Then
for any integer $m>0$, 
$$\int_{\XX(m)}|\omega(m)|=\ll^{-d}\sum_{\emptyset\neq J\subset \ii}(\ll-1)^{|J|-1}[\tilde{E}_{J}^{\circ}]\left(\sum_{\substack{k_i\geq 1, i\in J\\
\sum_{i\in J}k_i N_i=d}}\ll^{-\sum_{i}k_i\mu_i}\right)\in \mM_{\XX_0}^{\mu_m}.$$
\end{thm}

Furthermore, from \cite[Corollary 7.13]{Nicaise} we have:
\begin{prop}\label{prop:motivic:volume:result}
With the same notations and conditions as in Theorem \ref{motivic:integration:formula:omegam}, 
the volume Poincar\'e series $S(\XX,\omega;T)$ is rational over $\mM_{\XX_0}$. In fact, 
let $\mu_i:=\ord_{E_i}\omega$, then
$$S(\XX,\omega;T)=\ll^{-d}
\sum_{\emptyset\neq J\subset \ii}(\ll-1)^{|J|-1}[\tilde{E}_{J}^{\circ}]\prod_{i\in J}\frac{\ll^{-\mu_i}T^{N_i}}{1-\ll^{-\mu_i}T^{N_i}}\in\mM_{\XX_0}^{\hat{\mu}}[\![T]\!].$$

The limit 
$$S(\XX,\widehat{\kk}^{s}):=-\lim_{T\to \infty}S(\XX,\omega;T):=\ll^{-d}\sS_{f}$$
is called the \textbf{motivic volume} of $\XX$, where 
$$\sS_{f}=\sum_{\emptyset\neq J\subset \ii}(\ll-1)^{|J|-1}[\tilde{E}_{I}^{\circ}].$$  
 And 
\begin{align*}S(\XX_\eta,\widehat{\kk}^{s}):&=-\lim_{T\to \infty}S(\XX_\eta,\omega;T)=-\lim_{T\to \infty}\sum_{m\geq 1}\left(\int_{\XX_{\eta}}|\omega(m)|\right)T^m\\
&=\ll^{-d}\int_{\XX_0}\sS_{f}\in\mM^{\hat{\mu}}_{\kappa}
\end{align*}
is called the \textbf{motivic volume} of $\XX_\eta$.
\end{prop}

\sss 
Let $(\XX, f)$ be a generically smooth formal $R$-scheme.  From Proposition \ref{prop:motivic:volume:result}, the motivic vanishing cycle 
$\sS_{f}$ belongs to $\mM_{\XX_0}^{\hat{\mu}}$.  For any point $x\in \XX_0$, let 
$$\sS_{f,x}=\sum_{\emptyset\neq J\subset \ii}(\ll-1)^{|J|-1}[\tilde{E}_{I}^{\circ}\cap h^{-1}(x)],$$  
where $h: \XX^\prime\to \XX$ is the resolution of singularities. We call $\sS_{f,x}$ the motivic Milnor fiber of $x\in\XX_0$.

\sss In summary, if we let $K(\GBSRig_{\kk})$ be the Grothendieck ring of the category of gauge bounded smooth rigid $\kk$-varieties. 
Here for an object $\XX_\eta$ in $\GBSRig_{\kk}$ we understand that the rigid variety $\XX_\eta$ 
comes from the generic fiber of a generically smooth special formal $R$-scheme $f: \XX\to \spf(R)$ with gauge bounded form 
$\omega$.  The Grothendieck ring 
$$K(\GBSRig_{\kk}):=\bigoplus_{d\geq 0}K(\GBSRig_{\kk}^d)$$
is defined in \cite[\S 5.2]{Le}. 

Let $K(\BSRig_{\kk})$ be the Grothendieck ring of the category $\BSRig_{\kk}$ of bounded smooth rigid $\kk$-varieties, which is obtained from $K(\GBSRig_{\kk})$ by forgetting the gauge form.  Then we can represent the above results in \S (\ref{motivic:result:formal:scheme}) as follows:

\begin{thm}\label{MV}
There exists a homomorphism of additive groups:
$$\MV:  K(\BSRig_{\kk})\to \mM_{\kappa}^{\hat{\mu}}$$
given by:
$$[\XX_\eta]\mapsto S(\XX_\eta, \widehat{\kk}^{s})$$
for a generically smooth special formal $R$-scheme $\XX$.  Moreover, if $\XX$ has relative dimension $d$, then 
$$\MV([\XX_\eta])=\ll^{-d}\cdot \int_{\XX_0}\sS_{f}\in \mM_{\kappa}^{\hat{\mu}}.$$
So $\MV$ is a morphism from the group $K(\BSRig_{\kk})$ to the group $\mM_{\kappa}^{\hat{\mu}}$. 

Moreover,  if $x\in\XX_0$ and let 
$$\hat{f}_{x}: \spf(\widehat{\oO}_{\XX,x})\to\spf(R)$$
be the formal completion of $\XX$ along $x$, then the generic fiber $\spf(\widehat{\oO}_{\XX,x})_{\eta}$ of the formal completion is the analytic Milnor fiber $\FF_x(\hat{f})$ defined in \ref{analytic:Milnor:fiber}, and 
$$\MV([\FF_x(\hat{f})])=\ll^{d}\cdot \sS_{f,x}.$$
\end{thm}

\section{Proof of Theorem \ref{main:thm}}\label{proof:conjecture}

\subsection{Techniques on the motivic constructible functions of Cluckers and Loeser}\label{CL:motivic:constructible}
\sss
In this section we learn a little bit about Cluckers-Loeser's  motivic constructible  function theory in \cite{CL}, which we will use to prove the conjecture. 
Le in \cite{Le} uses another method of Hrushovski-Kazhdan's ACVF theory in \cite{HK} to prove the Kontsevich-Soibelman conjecture on the motivic Milnor fiber.  Later on he can use the theory of motivic constructible functions of  Cluckers-Loeser to give a new proof, which is working over any field of characteristic zero.  We adapt such a beautiful theory to our applications for the motivic Joyce-Song formula.

\sss  The theory of motivic constructible  functions is motivated by the constructible functions for the Euler characteristic over reals.  The idea of Cluckers-Loeser is to do integration (functions defined on) subobjects of $\kappa(\!(t)\!)^m$, or more wiser, integration on subobjects of 
$$\kappa(\!(t)\!)^m\times\kappa^n\times\zz^r.$$
The theory is based on the Denef-Pas language $\lL_{\DP}$ with the ring language for valued fields and residue fields and with the Presburger language for valued groups.  Let $\T$ be the theory of algebraic closed fields containing $\kappa$, as in \cite[\S 16.2, 16.3]{CL}, then 
$(\kk(\!(t)\!), \kk, \zz)$ is a model of $\T$.  The primary definable $\T$-subassignment has the forms:
$$h[m,n,r](\kk):=\kk(\!(t)\!)^m\times\kk^n\times\zz^r.$$
It can be taken as a functor 
$$h[m,n,r]: \kk\supset \kappa\to \text{Category of sets}.$$
Any formula $\varphi$ in $\lL_{\DP}$ with coefficients in $\kappa(\!(t)\!)$, and coefficients in $\kappa$, defines a subassignment 
$h_{\varphi}\subset h[m,n,r]$ by:
$$h_{\varphi}(\kk)=\{x\in h[m,n,r](\kk)| (\kk, \kk(\!(t)\!),\zz)=\varphi(x)\}.$$
More generally, if $W=\XX\times X\times\zz^r$, with $\XX$ a $\kappa(\!(t)\!)$-variety, $X$ a $\kappa$-variety, then 
$$h_{W}(\kk):=\XX(\kk(\!(t)\!))\times X(\kk)\times\zz^r.$$

\begin{defn}
We define $\Def_{\kappa}$ to be the category of all the definable $\T$-subassignments
$$\kk\mapsto h_{\varphi}(\kk).$$
Let $S\subset \Def_{\kappa}$ be any object. Let $\Def_{S}$ ~\textbf{or}~ ($\Def_{S}(\lL_{\DP},\T)$) be the category of objects of $\Def_{\kappa}$ over $S$.   Define
$$\RDef_{S};~\textbf{or}~ (\RDef_{S}(\lL_{\DP},\T))$$
to be the subcategory of $\Def_{S}$ whose objects are subassignments of $S\times h_{\aaa_{\kappa}^n}$, for variable
$n$, morphisms to $S$ are the ones induced by the projection onto the $S$-factor.
\end{defn}

We define the Grothendieck group for $\RDef_{S}$.
\begin{defn}\label{Grothendieck:group:RDef}
The Grothendieck group $K_0(\RDef_{S})$ is defined to be a free abelian group generated by symbols:
$$[X\to S]$$
with $X\to S$ in $\RDef_{S}$, modulo the relations:
$$[X\to S]=[Y\to S]$$
if $[X\to S]$ is isomorphic to $[Y\to S]$ in $\RDef_{S}$, and 
$$[X\cup Y\to S]+[X\cap Y\to S]=[X\to S]+[Y\to S]$$
for any definable $\T$-subassignments $X$ and $Y$ of $S\times h_{\aaa_{\kappa}^n}$ for some $n\in \nn$.
\end{defn}

\sss Let $X\to S$ be an object in  $\RDef_{S}$ and $m\in\nn_{>0}$.  Assume that $X=h_{W}$ with 
$W=\XX\times X\times \zz^r$.  A good $\mu_m$-action on $X$ is a $\mu_m$-action 
$$\mu_m\times X\to X$$
on $X$ such that each orbit intersected with $h_{X}$ is contained in $h_{V}$ with $V$ an affine subvariety of $X$. 
A good $\hat{\mu}$-action on $X$ is a $\hat{\mu}$-action on $X$ that factors through a good $\mu_m$-action on $X$ for some 
$m\in\nn_{>0}$. 

\begin{defn}
The monodromic Grothendieck group $K^{\hat{\mu}}_0(\RDef_{S})$ is a free abelian group generated by:
$$[X\to S, \hat{\mu}]; ([X,\hat{\mu}])$$
with $X\to S$ in $\RDef_{S}$, and $X$ admits a $\hat{\mu}$-action, with the relations in Definition \ref{Grothendieck:group:RDef}, together with one more relation: 
$$[X\times h_{V}, \hat{\mu}]=[X\times h_{\aaa_{\kappa}^n}, \hat{\mu}],$$
where $V$ is the $n$-dimensional affine $\kappa$-space endowed with an linear 
$\hat{\mu}$-action and $\aaa^n_{\kappa}$ with trivial $\hat{\mu}$-action for $n\in\nn$.
\end{defn}

The groups $K_0(\RDef_{S})$ and $K^{\hat{\mu}}_0(\RDef_{S})$ are rings with respect to the fiber product 
of subassignments in \cite[\S 2.2]{CL}.

\sss We talk about the rings of motivic constructible functions.  Let 
$$A:=\zz[\ll, \ll^{-1}, (1-\ll^i)^{-1}, i>0],$$
where $\ll$ is the Lefschetz motive of the affine line $\aaa^{1}_{\kappa}$.  For $S\in\Def_{\kappa}$, let $\pP(S)$ be the subring of the ring of functions 
$$S\to A$$
generated by:
\begin{enumerate}
\item all constant functions into $A$;
\item all definable functions $S\to \zz$;
\item all functions of the form $\ll^{\alpha}$, where $\alpha: S\to\zz$ is a definable function. 
\end{enumerate}
This is called the ring of Presburger functions as in \cite{CL}. 

Let $\pP^0(S)$ be the subring of $\pP(S)$ generated by $\ll-1$ and by character function $\mathds{1}_{Y}$ for all definable subassignments $Y$ of $S$. 

\begin{defn}
The ring of constructible motivic functions $\C(S)$ on $S$ and the monodromic one $\C^{\hat{\mu}}(S)$ are defined as:
$$\C(S):=K_0(\RDef_{S})\otimes_{\pP^0(S)}\pP(S); \quad  \C^{\hat{\mu}}(S):=K^{\hat{\mu}}_0(\RDef_{S})\otimes_{\pP^0(S)}\pP(S).$$
\end{defn}
The following result can be found in \cite{Le2}, \cite[\S 16.2, \S16.3]{CL}.

\begin{prop}
Let $X$ be an algebraic variety. Then 
\begin{enumerate}
\item $K_0(\RDef_{h_{X}})\cong K_0(\Var_{X})$;
\item $\C(h_{X})\cong \mM_{X,\loc}$;
\item $K^{\hat{\mu}}_0(\RDef_{h_{X}})\cong K^{\hat{\mu}}_0(\Var_{X})$;
\item $\C^{\hat{\mu}}(h_{X})\cong \mM^{\hat{\mu}}_{X,\loc}$.
\end{enumerate}
\end{prop}

\subsection{The proof of Formula (1) in Theorem \ref{main:thm}}\label{proof:formula1}

\sss Recall that for coherent sheaves or semi-Schur objects $E_1, E_2, E:=E_1\oplus E_2\in D^b(\Coh(Y))$ for a smooth Calabi-Yau threefold $Y$,  we have the following data:
\begin{enumerate}
\item $f: \Ext^1(E,E)\to\aaa_{\kappa}^1$;
\item $f_1: \Ext^1(E_1,E_1)\to\aaa_{\kappa}^1$;
\item $f_2: \Ext^1(E_2,E_2)\to\aaa_{\kappa}^1$,
\end{enumerate}
where $f, f_1,f_2$ are the corresponding polynomial functions coming from the corresponding $d$-critical Artin stack structure in \cite{BBBJ}.

We have:
$$\Ext^1(E,E)=\Ext^1(E_1,E_1)\oplus \Ext^1(E_2,E_2)\oplus \Ext^1(E_1,E_2)\oplus \Ext^1(E_2,E_1)$$
and let $(x,y,z,w)$ be the corresponding coordinates of $\Ext^1(E,E)$.  Let $\cc^*$ act on $\Ext^1(E,E)$ by 
$$\lambda\cdot (x,y,z,w)=(x,y,\lambda\cdot z, \lambda^{-1}\cdot w).$$
Then $\Ext^{1}(E,E)^{\cc^*}=\Ext^1(E_1,E_1)\oplus \Ext^1(E_2,E_2)$.

\sss \label{special:formal:scheme:E} For any semi-Schur object  $E\in D^b(\Coh(Y))$, we have the motivic Milnor fiber 
$\sS_{0}(E)\in \mM_{\kappa}^{\hat{\mu}}$ defined in Definition \ref{intro:motivic:Milnor:fiber}.  It is given by the generically smooth special formal 
$R$-scheme 
$$\hat{f}:=\hat{f}_{E}:  \widehat{\Ext^1(E,E)}\to\spf(R)$$
which is the formal completion of $f:=f_E: \Ext^1(E,E)\to\kappa$ along the origin. 

\sss \label{sec:notations} We make the following notations:
$$X:=\Ext^1(E,E); \quad Z:=\Ext^1(E_1,E_1)\oplus\Ext^1(E_2,E_2).$$
For $E=E_1\oplus E_2$, let 
$$\hat{f}: \XX:=\widehat{X}\to\spf(R)$$
be the formal completion of $X$ along the origin.  Then 
$$\hat{f}|_{Z}: \ZZ\to\spf(R)$$
is the formal completion of $f|_{Z}$ along the origin. 
Let $d=\dim(X)$,  $d_i:=\dim_\kappa(\Ext^{1}(E_i, E_i))$ for $i=1,2$ and 
$d_{12}=\dim_\kappa(\Ext^{1}(E_1, E_2))$, $d_{21}=\dim_\kappa(\Ext^{1}(E_2, E_1))$. 

For the formal $R$-scheme $\XX$, we have the generic fiber:
\[
\XX_{\eta}= 
\left\{
(x,y,z,w)\in\aaa_{\kk}^{d,\an} \Big|
 \begin{array}{l}
  \text{$ \val(x)>0, \val(y)>0;$} \\
  \text{$ \val(z)>0, \val(w)>0;$} \\
  \text{$f(x,y,z,w)=t$.} 
 \end{array}\right\}
\]
Here $\val(x):=\min_{1\leq i\leq d_1}\{\val(x_i)\}$, and $\val(y), \val(z),\val(w)$ are similarly defined. 
We divide the generic fiber $\XX_\eta$ into two parts:
$$\XX_\eta=\X_0\sqcup\X_1,$$
where
$$\X_0=\{(x,y,z,w)\in \XX_\eta| z=0 \text{~or~} w=0\}$$
and
$$\X_1=\{(x,y,z,w)\in \XX_\eta| z\neq 0, w\neq 0\}=\XX_\eta\setminus \X_0.$$

\begin{thm}\label{formula:SfX:SfY}
We have:
$$\sS_{\hat{f},0}=\ll^{d-1}\cdot \MV([\XX_\eta])$$
and 
$$\sS_{\hat{f}|_{Z}, 0}=\ll^{d-1}\cdot \MV([\X_{0}]),$$
where 
$$\MV: K(\BSRig_{\kk})\to\mM_\kappa^{\hat{\mu}}$$
is the homomorphism in Theorem \ref{MV}. 
\end{thm}
\begin{proof}
The first formula is just from Theorem \ref{MV}.  We prove:
$$\sS_{\hat{f}|_{Z},0}=\ll^{d-1}\cdot \MV([\X_0]).$$
By the property of the potential function 
$f: X\to\kappa$, if $z=0$ or $w=0$, then $f(x,y,z,w)=f(x,y,0,0)$.  Then we may write:
$$\X_0=\Y_0\times \ZZ_\eta,$$
where
$$\Y_0=\{(z,w)\in\aaa^{d_3+d_4,\an}_{\kk}| \val(z)>0, \val(w)>0\}$$
and 
\[
\ZZ_{\eta}= 
\left\{
(x,y,0,0)\in\aaa_{\kk}^{d_1+d_2,\an} \Big|
 \begin{array}{l}
  \text{$ \val(x)>0, \val(y)>0;$} \\
  \text{$f(x,y,0,0)=t$.} 
 \end{array}\right\}
\]
Let $dz\wedge dw:=dz_1\wedge\cdots\wedge dz_{d_3}\wedge dw_1\wedge\cdots\wedge dw_{d_4}$ be the standard gauge form on the open ball 
$\Y_0$. From Theorem \ref{motivic:integration:formula:omegam}, 
$$\int_{\Y_0(m)}|dz\wedge dw(m)|=\ll^{-d_3-d_4}.$$
So
\begin{align*}
\MV([\X_0])=\MV([\Y_0\times\ZZ_\eta])&=-\lim_{T\to \infty}\sum_{m\geq 1}\left(\int_{\Y_0(m)\times\ZZ_\eta(m)}|dz\wedge dw\wedge \omega(m)|\right)T^m \\
&=-\ll^{-d_3-d_4}\lim_{T\to \infty}\sum_{m\geq 1}\left(\int_{\ZZ_\eta(m)}|\omega(m)|\right)T^m \\
&=\ll^{1-d}\sS_{\hat{f}|_{Y},0}.
\end{align*}
So
$$\ll^{d-1}\cdot\MV([\X_0])=\sS_{\hat{f}|_{Y},0}.$$
\end{proof}

\sss We prove the following result:

\begin{thm}\label{MVX1=0}
We have
$$\MV([\X_1])=0$$
in $\mM_{\kappa,\loc}^{\hat{\mu}}$.
\end{thm}
\begin{proof}
The formal scheme 
$\hat{f}: \XX=\widehat{X}\to \spf(R)$
 is the $t$-adic formal completion of $f: \Ext^1(E,E)\to \aaa_{\kappa}^1$ along the origin.  Using Definition \ref{defn_Poincare_series}, Proposition \ref{prop:motivic:volume:result} and Theorem \ref{MV}, we have
 \begin{align*}
 \MV([\XX_{\eta}])&=\ll^{-d}\int_{\XX_0}\sS_{f}\\
 &=\ll^{-d}\int_{\XX_0} -\lim_{T\to \infty} \sS(\XX,\omega, T)\\
 &=-\lim_{T\to \infty}\sum_{n>0}\ll^{-nd}\int_{\XX_0}\XX_n(f)\\
 &=-\lim_{T\to \infty}\sum_{n>0}\ll^{-nd}([\X_0(n)]+[\X_1(n)])
 \end{align*}
where 
$$\X_1(n)=\{(x,y,z,w\in\aaa_{\kk}^{d, \an})| z\neq 0, w\neq 0\}.$$
As in \cite[\S 3]{Le22}, we have 
$$\X_1(n)=\XX_n(f)\cap (\lL_n(\aaa_{\kappa}^{d_1+d_2}\times \{0\}\times\aaa_{\kappa}^{d_{21}})\cup 
\lL_n(\aaa_{\kappa}^{d_1+d_2}\times\aaa_{\kappa}^{d_{12}}\times \{0\}))$$
where $\lL(\cdot)$ is the jet scheme. Now we prove $\MV([\X_1(n)])=0$.

This is from the argument as in \cite[\S 3]{Le22}.  Let $\X_1(n,m)$ be the set of 
$(x,y,z,w)$ in $\X_1(n)$ with $\ord_{t}(z)+\ord_{t}(w)=m$. This set is stable under the $\mu_n$-action.  Then using the theory of Cluckers-Loeser's motivic integration, $\mathbb{G}_{m,\kk}$ is an algebraic group and it acts on 
$$\sX:=\aaa_{\kk}^{d_1+d_2}\times_{\kk}\aaa_{\kk}^{d_{12}}\setminus \{0\}\times_{\kk}\aaa_{\kk}^{d_{21}}\setminus\{0\}$$
by 
$$\tau(x,y,z,w)=(x,y,\tau^{-1}z, \tau w).$$
Let $\yY$ be the quotient and $\phi: \sX\to\sY$ the projection.  Then $\phi$ induces a definable morphims 
$$h_{\phi}: h_{\sX}\to h_{\sY}$$
of definable subassignments in Cluckers-Loeser motivic integration theory.  In \cite[\S 3]{Le22}, Le found a surjective map 
$$\X_1(n,m)\to \widetilde{\X_1(n,m)}$$
with fibre over a point of residue field $k$ $\{\tau\in \lL(\aaa_{k}^1)| 0\leq \ord_{t}\tau <m\}$, such that 
$h_{\phi}(\X_1(n,m))=\widetilde{\X_1(n,m)}$.  Then Le proves that 
$$[\X_1(n,m)]=[\widetilde{\X_1(n,m)}]\cdot \ll^{n+1}\cdot (1-\ll^{-m})$$
which vanishes in $\mM_{\kappa, \loc}^{\hat{\mu}}$.
\end{proof}

\sss Now the proof of Formula (1) in Theorem \ref{main:thm} is obtained as follows:  From the formula $\sS_{\hat{f},0}=\ll^{d-1}\cdot \MV([\XX_\eta])$ in Theorem \ref{formula:SfX:SfY} and  $\MV([\X_1])=0$ in Theorem \ref{MVX1=0}, 
since
$\MV([\XX_\eta])=\MV([\X_0])+\MV([\X_1])$,  we have  $\sS_{\hat{f},0}=\ll^{d-1}\cdot \MV([\X_0])$. 
So by Theorem \ref{formula:SfX:SfY} again, 
$$\sS_{\hat{f},0}=\sS_{\hat{f}|_{Z},0}.$$

The potential function 
$$f|_{Z}=f_1+f_2, $$
where $f_1, f_2$ are the potential functions on $Z_1:=\Ext^1(E_1, E_1)$ and $Z_2:=\Ext^2(E_2, E_2)$. Hence 
$$\hat{f}|_{Z}:  \ZZ\to\spf(R)$$
can be split into $\hat{f}|_{Z}=\hat{f}_1+\hat{f}_2$ with 
$\hat{f}_i:  \ZZ_i\to\spf(R)$ the formal completion of $Z_i$ along the origin for $i=1,2$.  By motivic Thom-Sebastiani theorem proved in \cite{DeLo1} for regular functions, 
$$(1-\sS_{\hat{f}|_{Z},0})=(1-\sS_{\hat{f}_1,0})\cdot (1-\sS_{\hat{f}_2,0})$$
Note that $\sS_{0}(E)=\sS_{\hat{f},0}$, and $\sS_{0}(E_1)=\sS_{\hat{f}_1,0},  \sS_{0}(E_2)=\sS_{\hat{f}_2,0}$, we have
$$(1-\sS_{0}(E))=(1-\sS_{0}(E_1))\cdot (1-\sS_{0}(E_2)).$$

\subsection{The proof of  Formula (2) in Theorem \ref{main:thm}}\label{proof:formula2}

\sss We use the same notations as in \S (\ref{sec:notations}).  For the coherent sheaves or simple complexes 
$E=E_1\oplus E_2$, $E_1, E_2$,  let 
$$\hat{f}: \XX=\widehat{\Ext^{1}(E,E)}\to\spf(R)$$
be the special formal scheme as in \S (\ref{special:formal:scheme:E}). 
The formula  (2) in Theorem \ref{main:thm} is equivalent to the following formula:
\begin{align}\label{equivalent:formula:2}
& \int_{F\in\pp(\ext^1(E_2, E_1))} \sS_{0}(F)-\int_{F\in\pp(\ext^1(E_1, E_2))} \sS_{0}(F)  \\
&\nonumber =([\pp(\Ext^1(E_2, E_1))]-[\pp(\Ext^1(E_1, E_2))]) \cdot  \sS_{\hat{f}|_{Z},0}.
\end{align}

Recall that from \cite[Theorem 1.1]{Le}, \cite{Le2},  we have 
$$\int_{F\in\ext^1(E_2, E_1)} \sS_{0}(F)=\ll^{\dim\ext^1(E_2, E_1)}\sS_{\hat{f}|_{Z},0}$$
$$\int_{F\in\ext^1(E_1, E_2)} \sS_{0}(F)=\ll^{\dim\ext^1(E_1, E_2)}\sS_{\hat{f}|_{Z},0}$$
Hence formula (\ref{equivalent:formula:2}) is just from these two results. 
$\square$


\appendix

\section{The  algebra homomorphism}\Label{Application:homo}

In this appendix we study the  algebra homomorphism from the motivic Hall algebra of the abelian category of coherent sheaves on the 
Calabi-Yau threefold $Y$ to the motivic quantum torus.

\subsection{Motivic Hall algebras}\label{sec:motivic:Hall}

 In this section we review the definition and construction of motivic Hall algebra of Joyce and Bridgeland in \cite{Joyce07}, \cite{Bridgeland10}. 
We define the Grothendieck ring of stacks of finite type. 
\begin{defn}
The Grothendieck ring of stacks  $K(\St/\kappa)$ is defined to be the $\kappa$-vector space spanned by isomorphism classes of Artin stacks of finite type over $\kappa$ with affine stabilizers, modulo the relations:
\begin{enumerate}
\item for every pair of stacks $\sX_1$ and $\sX_2$ a relation:
$$[\sX_1\sqcup\sX_2]=[\sX_1]+[\sX_2];$$
\item for any geometric bijection $f: \sX_1\to \sX_2$, $[\sX_1]=[\sX_2]$;
\item for any Zariski fibrations $p_i: \sX_i\to \sY$ with the same fibers, $[\sX_1]=[\sX_2]$.
\end{enumerate}
\end{defn}
Let $[\aaa^1]=\ll$, the Lefschetz motive.  If $S$ is a stack of finite type over $\kappa$, we define the relative Grothendieck ring of stacks $K(\St/S)$ as follows:

\begin{defn}
The relative Grothendieck ring of stacks  $K(\St/S)$ is defined to be the $\kappa$-vector space spanned by isomorphism classes of morphisms
$$[\sX\stackrel{f}{\rightarrow}S],$$
with $\sX$ an Artin stack over $S$ of finite type with affine stabilizers, modulo the following relations:
\begin{enumerate}
\item for every pair of stacks $\sX_1$ and $\sX_2$ a relation:
$$[\sX_1\sqcup\sX_2\stackrel{f_1\sqcup f_2}{\longrightarrow}S]=[\sX_1\stackrel{f_1}{\rightarrow}S]+[\sX_2\stackrel{f_2}{\rightarrow}S];$$
\item for any diagram:
\[
\xymatrix{
\sX_1\ar[rr]^{g}\ar[dr]_{f_1}&&\sX_2\ar[dl]^{f_2}\\
&S,&
}
\]
where $g$ is a 
geometric bijection, then $[\sX_1\stackrel{f_1}{\rightarrow}S]=[\sX_2\stackrel{f_2}{\rightarrow}S]$;
\item for any pair of Zariski fibrations 
$$\sX_1\stackrel{h_1}{\rightarrow} \sY;\quad \sX_2\stackrel{h_2}{\rightarrow} \sY$$
with the same fibers, and $g: \sY\to S$, a relation 
$$[\sX_1\stackrel{g\circ h_1}{\longrightarrow} S]=[\sX_2\stackrel{g\circ h_2}{\longrightarrow} S].$$
\end{enumerate}
\end{defn}

The motivic Hall algebra  in \cite{Bridgeland10} is defined as follows.
Let $\MM$ be the moduli stack of coherent sheaves on $Y$. It is an algebraic stack, locally of finite type over $\kappa$. The motivic Hall algebra is the vector space 
$$H(\A)=K(\St/\MM)$$
equipped with a non-commutative product given by the role:
$$[\sX_1\stackrel{f_1}{\longrightarrow} \MM]\star[\sX_2\stackrel{f_2}{\longrightarrow} \MM]=[\zZ\stackrel{b\circ h}{\longrightarrow} \MM],$$
where $h$ is defined by the following Cartesian square:
\[
\xymatrix{
\zZ\ar[r]^{h}\ar[d]&\MM^{(2)}\ar[r]^{b}\ar[d]^{(a_1,a_2)}&\MM \\
\sX_1\times\sX_2\ar[r]^{f_1\times f_2}& \MM\times \MM,&
}
\]
with $\mM^{(2)}$ the stack of short exact sequences in $\A$, and 
the maps $a_1, a_2, b$ send a short exact sequence
$$0\rightarrow A_1\longrightarrow B\longrightarrow A_2\rightarrow 0$$
to sheaves $A_1$, $A_2$, and $B$ respectively. Then $H(\A)$ is an algebra over 
$K(\St/\kappa)$. 

\subsection{Algebraic d-critical locus}\Label{sec:algebraic:d:critical:locus}
 
\sss We recall the algebraic (analytic) $d$-critical locus from \cite{Joyce}.  We mainly focus on the algebraic version, and  the analytic version can be  slightly modified from the  algebraic one. 

The algebraic $d$-critical locus is a classical model for the $-1$-shifted symplectic derived scheme as developed by PTVV in \cite{PTVV}. In the same paper \cite{PTVV}, PTVV prove that the moduli space of stable coherent sheaves or simple complexes over Calabi-Yau threefolds admits a 
$-1$-shifted symplectic derived structure, hence their underlying moduli  scheme has an algebraic $d$-critical locus structure. 
Thus the algebraic $d$-critical locus of Joyce provides the classical schematical framework for the moduli space of stable simple complex over smooth Calabi-Yau threefolds. 

\sss To define the algebraic $d$-critical locus, we first recall the following theorem in \cite{Joyce}:

\begin{thm}\label{property:local:data}(\cite{Joyce})
Let $X$ be a $\kappa$-scheme, which is locally of finite type. Then there exists a sheaf $\sS_{X}$ of $\kappa$-vector spaces on $X$, unique up to canonical isomorphism, which is uniquely characterized by the following two properties:

(i) Suppose that $R\subseteq X$ is Zariski open, $U$ is a smooth $\kappa$-scheme, and $i: R\hookrightarrow U$ is a closed embedding. Then there is an exact sequence of sheaves of $\kappa$-vector spaces on $R$:
$$0\rightarrow I_{R,U}\longrightarrow i^{-1}(\oO_{U})\stackrel{i^{\#}}{\longrightarrow}\oO_{X}|_{R}\rightarrow 0,$$
where $\oO_{X}, \oO_{U}$ are the structure sheaves of $X$ and $U$, and $i^{\#}$ is the morphism of sheaves over $R$. 
There is also an exact sequence of sheaves of $\kappa$-vector spaces over $R$:
$$0\rightarrow\sS_{X}|_{R}\stackrel{\iota_{R,U}}{\longrightarrow}\frac{i^{-1}(\oO_{U})}{I^{2}_{R,U}}\stackrel{d}{\longrightarrow}
\frac{i^{-1}(T^*U)}{I_{R,U}\cdot i^{-1}(T^*U)}$$
where $d$ maps $f+I_{R,U}^2$ to $df+I_{R,U}\cdot i^{-1}(T^*U)$. 

(ii) If $R\subseteq S\subseteq X$ are Zariski open, and $U, V$ are smooth $\kappa$-schemes, and 
$$i: R\hookrightarrow U$$
$$j: S\hookrightarrow V$$
are closed embeddings. Let 
$$\Phi: U\to V$$ be a morphism with $\Phi\circ i=j|_{R}: R\to V$.  Then the following diagram of sheaves on $R$ commutes:
\begin{equation}\label{diagram:Algebraic:Dcritical:Locus}
\begin{CD}
0 @ >>>\sS|_{R}@ >{\iota_{S,V}|_{R}}>> \frac{j^{-1}(\oO_{V})}{I^2_{S,V}}|_{R}@ >{d}>>
\frac{j^{-1}(T^*V)}{I_{S,V}\cdot j^{-1}(T^*V)}|_{R}@
>>> 0\\
&& @VV{\id}V@VV{i^{-1}(\Phi^{\#})}V@VV{i^{-1}(d\Phi)}V \\
0@ >>> \sS_{X}|_{R} @ >{\iota_{R,U}}>>\frac{i^{-1}(\oO_{U})}{I^2_{R,U}}@ >{d}>> \frac{i^{-1}(T^*U)}{I_{R,U}\cdot i^{-1}(T^*U)} @>>> 0.
\end{CD}
\end{equation}
Here $\Phi: U\to V$ induces
$$\Phi^{\#}: \Phi^{-1}(\oO_{V})\to\oO_{U}$$
on $U$, and we have:
\begin{equation}\label{map1}
i^{-1}(\Phi^{\#}): j^{-1}(\oO_{V})|_{R}=i^{-1}\circ\Phi^{-1}(\oO_{V})\to i^{-1}(\oO_{U}),
\end{equation}
a morphism of sheaves of $\kappa$-algebras on $R$.  As 
$\Phi\circ i=j|_{R}$, then (\ref{map1}) maps to $I_{S,V}|_{R}\to I_{R,U}$,  and 
$I^2_{S,V}|_{R}\to I_{R,U}^{2}$.  Thus (\ref{map1}) induces the morphism in the second column of (\ref{diagram:Algebraic:Dcritical:Locus}). Similarly,  $d\Phi: \Phi^{-1}(T^*V)\to T^*U$ induces the third column of 
(\ref{diagram:Algebraic:Dcritical:Locus}).
\end{thm}

According to \cite{Joyce}, there is a natural decomposition 
$$\sS_{X}=\sS_{X}^0\oplus \kappa_{X}$$
and $\kappa_{X}$ is the constant sheaf on $X$ and $\sS_{X}\subset \sS_{X}$ is the kernel of the composition:
$$\sS_{X}\to\oO_{X}\stackrel{i_{X}^{\#}}{\longrightarrow}\oO_{X^{\red}}$$
with $X^{\red}$ the reduced $\kappa$-scheme of $X$, and $i_{X}: X^{\red}\hookrightarrow X$ the inclusion. 

\begin{defn}\label{algebraic:d:critical:locus}
An {\em algebraic d-critical locus} over the field $\kappa$ is a pair $(X,s)$, where $X$ is a $\kappa$-scheme, locally of finite type, and 
$s\in H^0(\sS_{X}^0)$ for $\sS_{X}^{0}$ in Theorem \ref{property:local:data}. These data satisfy the following conditions:
for any $x\in X$, there exists a Zariski open neighbourhood $R$ of $x$ in $X$, a smooth $\kappa$-scheme $U$, a regular function 
$f: U\to\kappa$, and a closed embedding $i: R\hookrightarrow U$, such that $i(R)=\Crit(f)$ as $\kappa$-subschemes of $U$, 
and $\iota_{R,U}(s|_{R})=i^{-1}(f)+I^2_{R,U}$.  We call the quadruple $(R,U,f,i)$ a {\em critical chart} on $(X,s)$. 
\end{defn}

\sss Some properties of $(X,s)$ are as follows:

\begin{thm}\cite{Joyce}
Let $(X,s)$ be an algebraic $d$-critical locus, and 
$(R,U,f,i), (S, V,g,j)$ be critical charts on $(X,s)$.  Then for each $x\in R\cap S\subset X$ there exists subcharts 
$$(R^\prime, U^\prime, f^\prime, i^\prime)\subseteq (R, U, f, i),$$
$$(S^\prime, V^\prime, g^\prime, j^\prime)\subseteq (S, V, g, j)$$
with $x\in R^\prime\cap S^\prime\subseteq X$, a critical chart $(T, W,h,k)$ on $(X,s)$, and embeddings
$$\Phi: (R^\prime, U^\prime, f^\prime, i^\prime)\hookrightarrow (T, W, h, k)$$ and 
$$\Psi: (S^\prime, V^\prime, g^\prime, j^\prime)\hookrightarrow (T, W, h, k).$$
\end{thm}

\sss We introduce the canonical line bundle of $(X,s)$:

\begin{thm}\label{Canonical:line:bundle:Xs}\cite[Theorem 2.28]{Joyce}
Let $(X,s)$ be an algebraic $d$-critical locus, and $X^{\red}\subset X$ the associated reduced $\kappa$-scheme. 
Then there exists a line bundle $K_{X,s}$ on $X^{\red}$ which we call the {\em canonical Line bundle} of $(X,s)$, that is natural up to  canonical isomorphism, and is characterized by the following properties:

(i) If $(R,U,f,i)$ is a critical chart on $(X,s)$, there is a natural isomorphism 
$$\iota_{R,U,f,i}: (K_{X,s})|_{R^{\red}}\to i^*(K_{U}^{\otimes 2})|_{R^{\red}}$$
where $K_U$ is the canonical line bundle of $U$.

(ii) Let $\Phi: (R,U,f,i)\hookrightarrow  (S, V,g,j)$ be an embedding of critical charts on $(X,s)$. Then there is an 
isomorphism of line bundles on $\Crit(f)^{\red}$:
$$J_{\Phi}: (K_{U}^{\otimes 2})|_{\Crit(f)}\stackrel{\cong}{\longrightarrow} \Phi|_{\Crit(f)^{\red}}^{*}(K_{V}^{\otimes 2}).$$
Since $i: R\to \Crit(f)$ is an isomorphism as schemes with $\Phi\circ i=j|_{R}$, this gives 
$$i|_{R^{\red}}^{*}(J_{\Phi}): i^{*}(K_{U}^{\otimes 2})|_{R^{\red}}\stackrel{\cong}{\longrightarrow} j^{*}(K_{V}^{\otimes 2})|_{R^{\red}},
$$
and we have:
$$\iota_{S,V,g,j}|_{R^{\red}}=i|_{R^{\red}}^{*}(J_{\Phi})\circ \iota_{R,U,f,i}: (K_{X,s})|_{R^{\red}}
\to j^{*}(K_{V}^{\otimes 2})|_{R^{\red}}.$$
\end{thm}

\sss  We talk about the orientation data for $d$-critical locus in \cite{Joyce}:

\begin{defn}\label{oriented:d:critical:locus}
Let $(X,s)$ be an algebraic $d$-critical locus, and $K_{X,s}$ the canonical line bundle of $(X,s)$. 
An {\em orientation} on $(X,s)$ is a choice of square root line bundle $K_{X,s}^{1/2}$ for $K_{X,s}$ on $X^{\red}$. 
I.e., an orientation of $(X,s)$ is a line bundle $L$ over $X^{\red}$ and an isomorphism $L^{\otimes 2}=L\otimes L\cong K_{X,s}$. 
A $d$-critical locus with an orientation will be called an oriented $d$-critical locus.
\end{defn}

Bussi, Brav and Joyce \cite{BBJ} prove the following interesting result:  Let $(\bX, \omega)$ be a $(-1)$-shifted symplectic derived scheme over $\kappa$ in the sense of \cite{PTVV}, and let $X:=\bt_0(\bX)$ be the associated classical $\kappa$-scheme of 
$\bX$. Then $X$ naturally extends to an algebraic $d$-critical locus $(X,s)$. The canonical line bundle $K_{X,s}\cong \det(\ll_{\bX})|_{X^{\red}}$ is the determinant line bundle of the cotangent complex $\ll_{\bX}$ of $\bX$.

\sss One of the applications of the $(-1)$-shifted symplectic derived scheme or stack is on moduli problems.  Let $Y$ be a smooth Calabi-Yau threefold over $\kappa$, and $X$ a classical moduli scheme of simple coherent sheaves in $\Coh(Y)$, the abelian category of coherent sheaves on $Y$.  Then in \cite{PTVV}, the authors prove that there is a natural $(-1)$-shifted derived scheme structure 
$\mathbf{X}$ on the moduli space $X$, such that if 
$$i: X\hookrightarrow \mathbf{X}$$
is the inclusion, then the pullback $i^{*}\ll_{\mathbf{X}}$ of the cotangent complex of $\mathbf{X}$ is a perfect obstruction theory of 
$X$, 
thus from the result in \cite{BBJ}, $X$ has an algebraic $d$-critical locus structure. 

Actually similar story holds for derived Artin stacks. Still let $Y$ be a smooth Calabi-Yau threefold over $\kappa$, and $\mM$ a moduli stack of simple coherent sheaves in $\Coh(Y)$.  Then in \cite{PTVV},  there is a natural $(-1)$-shifted derived stack structure 
$\bmM$ on the moduli space $\mM$, such that if 
$$i: \mM\hookrightarrow \bmM$$
is the inclusion, then the pullback $i^{*}\ll_{\bmM}$ of the cotangent complex of $\bmM$ is an obstruction theory of 
$\mM$, 
thus from the result in \cite{BBBJ}, $\mM$ has an algebraic $d$-critical stack structure. 

\sss The orientation of  $d$-critical locus or stacks has application to motivic Donaldson-Thomas theory. 
Let $X:=(X,s)$ be an algebraic $d$-critical locus, which is the moduli scheme of simple coherent sheaves or simple complexes over a smooth Calabi-Yau threefold $Y$.  
Let $(R,U,f,i)$ be a critical chart of $(X,s)$. Then in \cite{BJM}, the authors associated with this local chart a perverse sheaf of vanishing cycle
\begin{equation}\label{local:vanishing:cycle}
\sS_{U,f}^{\phi}\in \mM_{X}^{\hat{\mu}}
\end{equation}
such that 
$$\sS_{U,f}^{\phi}|_{X_c}=\ll^{-\dim(U)/2}\odot ([U_c, \hat{\iota}]-\sS_{U, f-c})|_{X_{c}}, $$
where $f: X\to\kappa$ is the function $f$ restricted to $X$, and 
$X=\sqcup_{c\inf(X)}X_c$ and $X_c=X\cap U_{c}$ with $U_{c}=f^{-1}(c)\in U$.  We call 
$\sS_{U,f}^{\phi}$ the {\em motivic vanishing cycle} of $f$. 

\sss As in \cite{BJM}, a principal $\zz_2$-bundle $P\to X$ is a proper, surjective, \'etale morphism of $\kappa$-schemes 
$\pi: P\to X$ together with a free involution $\sigma:P\to P$ such that the orbits of $\zz_2$ are the fibers of $\pi$. 

Let $\zz_2(X)$ be the abelian group of isomorphism classes $[P]$ of principal $\zz_2$-bundles $P\to X$, with multiplication
$[P]\cdot [Q]=[P\otimes_{\zz_2}Q]$ and the identity the trivial bundle $[X\times \zz_2]$.  We know that $P\otimes_{\zz_2}P\cong X\times\zz_2$, so every element in $\zz_2(X)$ has order $1$ or $2$. 

In \cite{BJM}, the authors define the motive of a principal $\zz_2$-bundle $P\to X$ by:
$$\Upsilon(P)=\ll^{-\frac{1}{2}}\odot([X,\hat{\iota}]-[P,\hat{\rho}])\in\mM_{X}^{\hat{\mu}},$$
where $\hat{\rho}$ is the $\hat{\mu}$-action on $P$ induced by the $\mu_2$-action on $P$. 

In \cite{BJM}, for any scheme $Y$, the authors define an ideal $I_{Y}^{\hat{\mu}}$ in $\mM_{Y}^{\hat{\mu}}$ which is generated by
$$\phi_{*}(\Upsilon(P\otimes_{\zz_2}Q)-\Upsilon(P)\odot\Upsilon(Q))$$
for all morphisms $\phi: X\to Y$ and principal $\zz_2$-bundles 
$P, Q$ over $X$.  Then define 
$$\overline{\mM}_{Y}^{\hat{\mu}}=\mM_{Y}^{\hat{\mu}}/I_{Y}^{\hat{\mu}}.$$
Then $(\overline{\mM}_{Y}^{\hat{\mu}},\odot)$ is a commutative ring with $\odot$ and there is a natural 
projection $\prod_{Y}^{\hat{\mu}}: \mM_{Y}^{\hat{\mu}}\to \overline{\mM}_{Y}^{\hat{\mu}}$.

\sss \label{sec:include:quadratic:form} Let $(X,s)$ be an oriented   $d$-critical locus. Recall the isomorphism in the canonical line bundle 
$K_{X,s}$ in Theorem \ref{Canonical:line:bundle:Xs}.  Let 
$Q_{R,U,f,i}\to R$ be the principal $\zz_2$-bundle parameterizing local isomorphisms
$$\alpha: K_{X,s}^{1/2}|_{R^{\red}}\to i^{*}(K_{U})|_{R^{\red}}$$
with $\alpha\otimes\alpha=\iota_{R,U,f,i}$, where 
$$\iota_{R,U,f,i}: K_{X,s}|_{R^{\red}}\to i^{*}(K_{U}^{\otimes 2})|_{R^{\red}}$$ is the isomorphism in  Theorem \ref{Canonical:line:bundle:Xs}.

\begin{thm}(\cite{BJM})\label{thm:global:motives}
If $(X,s)$ is a finite type algebraic $d$-critical locus with a choice of orientation $K_{X,s}^{1/2}$, then there exists a unique motive 
$$\sS_{X,s}^{\phi}\in\overline{\mM}_{X}^{\hat{\mu}}$$
such that 
if $(R,U,f,i)$ is a critical chart on $(X,s)$, then 
$$\sS_{X,s}^{\phi}|_{R}=i^{*}(\sS_{U,f}^{\phi})\odot \Upsilon(Q_{R,U,f,i})\in \overline{\mM}_{R}^{\hat{\mu}}$$
where 
$$\Upsilon(Q_{R,U,f,i})=\ll^{-1/2}\odot([R.\hat{\iota}]-[Q, \hat{\rho}])\in \mM_{R}^{\hat{\mu}}$$
is the motive of the principal $\zz_2$-bundle defined in 
\S 2.5 of \cite{BJM}. 
\end{thm}

\begin{rmk}
\begin{enumerate}
\item On the Donaldson-Thomas moduli scheme $X$ over the Calabi-Yau threefold $Y$, the Behrend function 
$$\nu_{X}: X\to\zz$$
is a constructible function defined using the local Euler obstruction of the canonical cycle of $X$ defined in \S 2 of \cite{Behrend}.  In \cite{Behrend}, Behrend proves 
$$\chi(X,\nu_{X})=\int_{[X]^{\virt}}1$$
if $X$ is a proper scheme, where $[X]^{\virt}$ is the virtual fundamental class of $X$ defined by the perfect obstruction theory.  This is the Donaldson-Thomas invariant. 

\item Let $(X,s)$ be the corresponding algebraic $d$-critical locus of the moduli scheme $X$. If $(X,s)$ is oriented, i.e. there exists a global square root
$K_{X,s}^{1/2}$, then there exists 
$\sS_{X}^{\phi}\in \overline{\mM}_{X}^{\hat{\mu}}$ such that 
$$\chi(X, \sS_{X}^{\hat{\mu}})=\chi(X,\nu_{X}),$$
thus categorifying the  Donaldson-Thomas invariant. 

\item The orientation data $K_{X,s}^{1/2}$ and the triangle property of the motives of the quadratic forms 
$Q$ were introduced by Kontsevich-Soibelman in \cite{KS} in the more general setting of the motivic Donaldson-Thomas invariants.  Several cases of the square root line bundle have been proved in \cite{Davison}, \cite{KL}, \cite{Hua}.
\end{enumerate}
\end{rmk}

\sss In order to define the  algebra homomorphism from the motivic Hall algebra to the motivic quantum torus, we need to modify the global motive 
$\sS_{X}^{\phi}$.

Let us now fix the moduli stack $\MM$ as the stack of coherent sheaves on the abelian category of coherent sheaves $\A$ of $Y$. From \cite{Bridgeland10}, \cite{JS}, $\MM$ is an Artin stack, locally of finite type. 

\begin{lem}
Let $(X,s)$ be a finite type  algebraic $d$-critical locus, which is the moduli scheme of stable coherent sheaves over 
$Y$, or the coarse moduli scheme of moduli of semi-stable coherent sheaves on $Y$, then $(X,s)$ is the coarse moduli scheme of an Artin stack $\sX$ of finite type, which is the underlying Artin stack of a 
$(-1)$-shifted derived Artin stack $\bmX$ in sense of \cite{BBBJ}.
\end{lem}
\begin{proof}
We consider the algebraic $d$-critical locus $(X,s)\subset \MM$ such that 
it is the coarse moduli space of a moduli stack $\sX$ with fixed topological data. 
From \cite{PTVV} and \cite{BBBJ}, $\sX$ can be extended to a canonical $(-1)$-shifted derived Artin stack structure 
$\bmX$. 
\end{proof}

\sss  On the Artin stack $\sX$, in \cite{BBBJ}, the authors define an algebraic $d$-critical stack structure $(\sX,s)$ on $\sX$, similar to Definition \ref{algebraic:d:critical:locus}.   An oriented algebraic $d$-critical stack is the one  $(\sX,s)$ such that  there exists a global square root line bundle
$K_{\sX,s}^{1/2}$. 
Let $t: X\to\sX$ be a morphism from an algebraic $d$-critical locus to the algebraic $d$-critical stack $(\sX,s)$, then in Theorem 5.14 of \cite{BBBJ}, there exists a 
$$\sS_{\sX}^{\phi}\in \overline{\mM}_{\sX,\loc}^{\hat{\mu}}$$
such that
$$t^{*}\sS_{\sX}^{\phi}=\ll^{n/2}\odot\sS_{X}^{\phi}\in \overline{\mM}_{X,\loc}^{\hat{\mu}},$$
where $n$ is the relative dimension of the morphism $t$.

\begin{rmk}
\'Etale locally if the algebraic $d$-critical stack  $(\sX,s)$ is given by the quotient stack 
$$[Q/E],$$
then $Q$ is an algebraic $d$-critical locus, and the morphism $t$ is given by:
$$t: Q\to [Q/E]$$
and 
$$\sS_{[Q/E]}^{\phi}=\ll^{\dim(E)/2}\odot\sS_{Q}^{\phi}$$
with $\sS_{Q}^{\phi}$ the motivic vanishing cycle sheaf. 
\end{rmk}

\subsection{The integration map}\label{motivic:integration:map}

\sss In this section we define the integration map from the motivic Hall algebra to the motivic quantum torus. 

\sss
Recall that in \S 3 of  \cite{Bridgeland10}, there exists maps of commutative rings:
$$K(\Sch/\kappa)\to K(\Sch/\kappa)[\ll^{-1}]\to K(\St/\kappa),$$
where $K(\Sch/\cc)$ is the Grothendieck  ring of schemes of finite type over $\kappa$. 
Since $H(\A)$ is an algebra over $K(\St/\kappa)$, define a $K(\Sch/\kappa)[\ll^{-1}]$-module
$$H_{\reg}(\A)\subset H(\A)$$
to be the span of classes of maps $[X\stackrel{f}{\rightarrow}\MM]$ with $X$ a scheme.  An element of $H(\A)$ is regular if it lies in $H_{\reg}(\A)$. Then from Theorem 5.1 of \cite{Bridgeland10}, the Hall algebra product preserves the regular elements in $H_{\reg}(\A)$. 

\sss For our purpose, we define a $K(\Sch/\kappa)[\ll^{-1}]$-module
$$H_{d-\Crit}(\A)\subset H(\A)$$
to be the span of classes of maps $[X\stackrel{f}{\rightarrow} \MM]$
with $(X,s)$ an ``oriented"  algebraic $d$-critical locus in the sense of Joyce \cite{Joyce}, reviewed in \S \ref{sec:algebraic:d:critical:locus}.  Since $X$ is a scheme, the module 
$$H_{d-\Crit}(\A)\subset H_{reg}(\A). $$
The following is a generalization of Theorem  5.1 of \cite{Bridgeland10}:

\begin{thm}\label{Thm:preserve:algebraic:Dcritical:submodule}
The sub-module of $d$-critical elements of $H(\A)$ is closed under the convolution product:
$$H_{d-\Crit}(\A)\star H_{d-\Crit}(\A)\subset H_{d-\Crit}(\A)$$
and is a $K(\Sch/\kappa)[\ll^{-1}]$-algebra. Moreover, the quotient 
$$H_{\ssc,d-\Crit}(\A)=H_{d-\Crit}(\A)/(\ll-1)H_{d-\Crit}(\A)$$
is a commutative $K(\Sch/\kappa)$-algebra.
\end{thm}
\begin{proof}
The proof is similar to Theorem  5.1 of \cite{Bridgeland10}. Let 
$$a_i=[X_i\stackrel{f_i}{\rightarrow}\MM]\in H_{d-\Crit}(\A); i=1,2$$
with both $X_1$ and $X_2$ are algebraic $d$-critical loci. Let $E_i$ be the family of coherent sheaves on $X_i$ corresponding to the map $f_i$. As in the proof of Theorem  5.1 of \cite{Bridgeland10}, stratify $X_1\times X_2$ by locally closed sub varieties $S_j$, we have the following diagram:
\[
\xymatrix{
\zZ_j\ar[r]\ar[d]_{t_j}&\zZ\ar[d]_{t}\ar[r]^{q}&\MM^{(2)}\ar[r]^{h}\ar[d]_{(a_1,a_2)}&\MM \\
S_j\ar[r]& X_1\times X_2\ar[r]^{f_1\times f_2}&\MM\times\MM, &
}
\]
where $\zZ_j$ is the fiber product.  Let
$$V^{k}(x_1,x_2)=\Ext^{k}_{\MM}(E_2|_{\{x_2\}\times\MM}, E_1|_{\{x_1\}\times\MM}); \quad (x_1, x_2)\in S$$
and $d_k(S_j)=\dim(\Ext^k_{\MM})$. Then from \S 7.1 of \cite{Bridgeland10}, 
$$\zZ_j\cong [Q_j/\kappa^{d_0(S_j)}],$$
where $Q_j=V^1(S_j)$ is the total space of the trivial vector bundle over $S_j$ with fiber $V^{1}(x_1, x_2)$ over 
$(x_1, x_2)\in S_j$, and 
$$a_1\star a_2=[\zZ\stackrel{b\circ h}{\rightarrow}\MM]=\sum_{j}\ll^{-d_{0}(S_j)}[Q_j\stackrel{g_j}{\rightarrow}\MM].$$
So for us we only need to prove that $Q_j$ is an algebraic $d$-critical locus. Since $Q_j$ is the trivial vector bundle 
$S_j\times\kappa^{d_1(S_j)}$, by assuming that $S_j$ is a strata such that $d_1(S_j)$ is constant, the algebraic $d$-critical structure on $S_j$ comes as follows:  around a point $(x_1, x_2)\in S_j$, there exists an algebraic function 
$$f_{E_1\oplus E_2}: \Ext^1(E_1\oplus E_2, E_1\oplus E_2)\to\kappa$$
such that $(x_1, x_2)\in \Crit(f_{E_1\oplus E_2})$.  Or we can use \cite[Corollary 5.17]{BBBJ} to argue that 
$\zZ_j=[Q_j/\kappa^{d_0(S_j)}]$ is a canonical truncation $\bt_0(\bmX)$ for a $(-1)$-shifted Artin stack 
$\bmX$, where $\bmX$ is the derived moduli stack of coherent sheaves of the form $E_1\oplus E_2$. 

The second statement is the same as in Theorem 5.1 of \cite{Bridgeland10}, since split $Q_j$ into the 
zero-section and the complement, we can write 
\begin{equation}\label{Hall:algebra:product1}
a_1\star a_2=\sum_{j}\ll^{-d_{0}(S_j)}\left([S_j\stackrel{k}{\rightarrow}\MM]+(\ll-1)[\pp(Q_j)\stackrel{g_j}{\rightarrow}\MM]\right)
\end{equation} 
and 
\begin{equation}\label{Hall:algebra:product2}
a_1\star a_2=\sum_{j}\ll^{-d_{0}(S_j)}[S_j\stackrel{k}{\rightarrow}\MM]=[X_1\times X_2\stackrel{k}{\rightarrow}\MM] \mod (\ll-1).
\end{equation}
So we are done. 
\end{proof}

\sss \label{Poisson:bracket:Hall:algebra} The algebra $H_{\ssc,d-\Crit}(\A)$ is called semi-classical Hall algebra for the elements of $d$-critical locus. In \cite{Bridgeland10}, Bridgeland also defines a Poisson bracket on $H(\A)$ by:
$$\{f, g\}=\frac{f\star g-g\star f}{\ll-1}.$$
This bracket preserves the subalgebra $H_{d-\Crit}(\A)$.

\sss  We define the motivic quantum torus. 
Let $K(Y)=K(\A)$ be the Grothendieck group of the category $\A$.  Let  $E, F\in k(\A)$ and let 
$$\chi(E,F)=\sum_{i}(-1)^{i}\dim_{\kappa}\Ext^i(E,F).$$
So $\chi(-,-)$ is a bilinear form on $K(\A)$, which is called the Euler form.  The numerical Grothendieck group is the quotient
$$N(Y)=K(Y)/K(Y)^{\perp},$$
where $K(Y)^{\perp}$ means the Euler form zero subgroup. 
Let $\Gamma\subset N(Y)$ denote the monoid of  effective classes, which is to say the classes of the form $[E]$ with $E$ a sheaf. 

\begin{rmk}
The stack $\MM$ split into disjoint union of open and closed substacks
$$\MM=\bigsqcup_{\alpha\in\Gamma}\MM_{\alpha}$$
where $\MM_{\alpha}\subset\MM$ is the stack of objects of class
$\alpha\in\Gamma$.  And 
$\MM_{\alpha}\subset \MM$ implies that $K(\St/\MM_{\alpha})\subset K(\St/\MM)$.

Also the Hall algebra
$$H(\A)=\bigoplus_{\alpha\in\Gamma}H(\A)_{\alpha}$$
and $H(\A)$ is a graded algebra with respect to the convolution product. 
\end{rmk}

\sss \label{Poisson:bracket:motivic:quantum:torus} Let $\mM_{\kappa}^{\hat{\mu}}$ be the ring of motives and consider
$$\mM_{\kappa, \loc}^{\hat{\mu}}=\mM_{\kappa}^{\hat{\mu}}[\ll^{-1}, \ll^{-1/2}, (\ll^i-1)^{-1}, i\in\nn_{\geq 0}].$$
Let $\overline{\mM}_{\kappa, \loc}^{\hat{\mu}}$ be the ring with the product $\odot$. 
\begin{defn}
Define
$$\overline{\mM}_{\kappa, \loc}^{\hat{\mu}}[\widehat{\Gamma}]=\bigoplus_{\alpha\in \Gamma}\overline{\mM}_{\kappa, \loc}^{\hat{\mu}}\cdot x^{\alpha}$$
to be the ring generated by symbols $x^\alpha$ for $\alpha\in \Gamma$, with product defined by:
$$x^\alpha\star x^{\beta}=\ll^{\frac{1}{2}\cdot\chi(\alpha,\beta)}\cdot x^{\alpha+\beta}.$$
Even the Euler form is skew-symmetric, this ring is not commutative due to the factor of power of the Lefschetz motive. 
\end{defn}

\begin{rmk}
One can try to define the Poisson bracket on 
$\overline{\mM}_{\kappa, \loc}^{\hat{\mu}}[\widehat{\Gamma}]$.
For classes $\alpha, \beta\in \Gamma$, we define 
$$\Ext_{E,F}^i(\alpha,\beta)=\Ext^i(E,F)$$
for $[E]=\alpha, [F]=\beta$, where $E, F\in \A$. 

Since different representatives $E^\prime, F^\prime$ of $\alpha,\beta$ may have different Extension groups, 
let $e^i:=\dim(\Ext^i(E,F))$ we calculate:
\begin{align*}
&\frac{\sum_{i=0}^{3}(-1)^{i+1}\ll^{\dim\ext_{E,F}^{i}(\alpha,\beta)}}{\ll-1} \\
&=\frac{-(\ll^{\chi(\alpha,\beta)}-1)}{\ll-1}\cdot 
\frac{\ll^{e^1}(\ll^{e^0-e^1-1}+\cdots+1)+\ll^{e^3}(\ll^{e^2-e^3-1}+\cdots+1)}
{\ll^{\chi(\alpha,\beta)-1}+\cdots+\ll+1} \\
&:=\frac{\ll^{\chi(\alpha,\beta)}-1}{\ll-1}\cdot  \mbox{Term}_{E,F}
\end{align*}
for any $E, F\in \A$ such that $[E]=\alpha$, $[F]=\beta$.
Let 
$$\overline{\mM}_{\kappa, \loc}^{\hat{\mu}}[\Gamma]=\overline{\mM}_{\kappa, \loc}^{\hat{\mu}}[\widehat{\Gamma}]_{\mbox{\tiny Term}_{E,F}}$$
be the localization ring on all such terms $\mbox{Term}_{E,F}$.  But the localization of noncommutative ring is very subtle.  

Then we use $\Ext^i(\alpha,\beta)$ to represent the extension group for any representatives. 
The Poisson bracket is given by:
\begin{align*}
&\{x^{\alpha}, x^{\beta}\}=\ll^{\frac{1}{2}\cdot\chi(\alpha,\beta)}\odot
\frac{\sum_{i=0}^{3}(-1)^{i+1}\ll^{\dim\ext^{i}(\alpha,\beta)}}{\ll-1} \cdot x^{\alpha+\beta}\\
&=\ll^{\frac{1}{2}\cdot\chi(\alpha,\beta)}\odot
\frac{\ll^{\chi(\alpha,\beta)}-1}{\ll-1} \cdot x^{\alpha+\beta}
\end{align*}
over $\overline{\mM}_{\kappa, \loc}^{\hat{\mu}}[\Gamma]$. 
\end{rmk}

We define the integration map. 
Let 
\begin{equation}\label{integration:map}
I: H_{\ssc, d-\Crit}(\A)\to \overline{\mM}_{\kappa, \loc}^{\hat{\mu}}[\Gamma]
\end{equation}
be the map defined by:
for any element 
$[Z\to\MM]\in H_{\ssc, d-\Crit}(\A)$, let 
$$t:Z\to\zZ$$
be the map from the oriented algebraic $d$-critical locus $Z$ to the corresponding oriented  $d$-critical Artin stack $\zZ$. Then 
$$I([Z\to\MM])=\left(\int t^{*}\sS_{\zZ}^{\phi}\right)\cdot x^{\alpha}\in \overline{\mM}_{\kappa, \loc}^{\hat{\mu}}[\Gamma]$$
where $\int: \overline{\mM}_{Z, \loc}^{\hat{\mu}}\to \overline{\mM}_{\kappa, \loc}^{\hat{\mu}}$ is the pushforward  of motives. 

\begin{rmk}
Let $\nu_{\zZ}$ be the Behrend function on $\zZ$ which is the pullback $i^*\nu_{\MM}$ from $i: \zZ\to\MM$.  Then taking cohomology of the perverse sheaf $t^{*}\sS_{\zZ}^{\phi}$ we get the weighted Euler characteristic 
$\chi(Z, t^*\nu_{\zZ})$.  This is the map $I$ in \cite[Theorem 5.2]{Bridgeland10}.
\end{rmk}

\begin{thm}\label{main:homomorphism}
The map $I$ in (\ref{integration:map}) is an algebra homomorphism. 
\end{thm}

\subsection{The proof of Theorem \ref{main:homomorphism}}\label{proof:homomorphism}

\sss For each algebraic $d$-critical locus $(Z,s)$ such that $Z$ factors through 
$[Z\stackrel{f}{\rightarrow}\MM_{\alpha}]$, the perverse sheaf $\sS^{\phi}_{Z}$ of vanishing cycles is constructible.  hence there exists a stratification $Z=\cup_{i}Z_i$ such that 
$\sS_{Z}^{\phi}|_{Z_i}$ is given by the vanishing cycle of the function 
$$f: U\to\kappa,$$
where we can take $Z_i$ fits into a critical chart $Z_i, U, f,i$. So $I$ is well-defined. 

\sss From Serre duality,
$$V^k(x_1, x_2)=V^{3-k}(x_2, x_1)^*. $$
Let $\hat{Q}_j=V^2(S_j)$ be the bundle over $S_j$ whose fiber at $(x_1,x_2)$ is 
$V^1(x_2, x_1)$.  Let 
$$g_j: Q_j\to\MM; \quad \hat{g}_j: \hat{Q}_j\to\MM$$
be the induced morphisms induced by taking the universal extensions. 
For 
$$\alpha_1=[X_1\stackrel{f_1}{\rightarrow}\MM_{\alpha_1}],$$
$$\alpha_2=[X_1\stackrel{f_2}{\rightarrow}\MM_{\alpha_2}],$$
we have:
\begin{equation}\label{Iai}
I(a_i)=\left(\int t_i^{*}\sS_{\sX_i}^{\phi}\right)\cdot x^{\alpha_i}\in \overline{\mM}_{\kappa, \loc}^{\hat{\mu}}[\Gamma],
\end{equation}
where 
$$t_i: X_i\to\sX_i$$
are the smooth morphisms from the $d$-critical loci to the corresponding $d$-critical Artin stacks for $i=1,2$. 

From the expression of $\alpha_1\star \alpha_2$ in (\ref{Hall:algebra:product2}) in the proof of Theorem \ref{Thm:preserve:algebraic:Dcritical:submodule}, 
\begin{equation}\label{Ia1stara2}
I(\alpha_1\star \alpha_2)=\left(\int\sum_{j}t_j^{*}\sS_{\sS_j}^{\phi}\right)\cdot x^{\alpha_1+\alpha_2}
=\left(\int t^{*}\sS_{\sX_1\times\sX_2}^{\phi}\right)\cdot x^{\alpha_1+\alpha_2},
\end{equation}
where $t_j:  S_j\to \sS_{j}$ is the morphism from the $d$-critical locus scheme to the $d$-critical stack
$\sS_{j}$, and so is the morphism
$$t:=t_1\times t_2: X_1\times X_2\to\sX_1\times \sX_2.$$ 
From \cite{BBBJ}, we have:
$$t^{*}\sS_{\sX_1\times\sX_2}^{\phi}=\ll^{n/2}\odot \sS_{X_1\times X_2}^{\phi}=
\ll^{n/2}\cdot \sS_{U,f}^{\phi}\odot\Upsilon(Q_{X,U,f,i}),$$
where $(X=X_1\times X_2, U, f,i)$ is the local critical  chart of $X_1\times X_2$, 
$$\sS_{U,f}^{\phi}=\ll^{-\dim(U)/2}(\mathds{1}_{U}-\sS_{U,f}),$$
 $n$ is the relative dimension of the smooth morphism 
$t$, and $\Upsilon(Q_{X,U,f,i})$ is the motive of a quadratic form, parameterizing the local isomorphism of the canonical line bundles as in 
\S (\ref{sec:include:quadratic:form}). 

Over a point $(x_1, x_2)\in X_1\times X_2\subset U$, the dimension 
$\dim(U)=\dim\Ext^1(E_1\oplus E_2, E_1\oplus E_2)$ and
$n=\dim\Ext^0(E_1\oplus E_2, E_1\oplus E_2)$, where $E_1, E_2$ are the  coherent sheaves corresponding
to $x_1, x_2$ respectively. 
When restricted to $X_1\times X_2\subset U$, the quadratic form 
$Q_{X,U,f,i}$ split into the product 
$$Q_{X, U, f,i}=Q_{X_1, U_1, f_1, i_1}\otimes Q_{X_2, U_2, f_2, i_2}$$
and 
$$\Upsilon(Q_{X,U,f,i})=\Upsilon(Q_{X_1,U_1,f_1,i_1})\odot \Upsilon(Q_{X_2,U_2,f_2,i_2})\in \overline{\mM}_{X_1\times X_2}^{\hat{\mu}}.$$
Here $(X_1, U_1, f_1,i_1)$ and $(X_2, U_2, f_2,i_2)$ are the critical charts of 
$x_1\in X_1$ and $x_2\in X_2$ respectively. 
For  coherent sheaves $E_1, E_2$, the first formula in Theorem \ref{main:thm} is:
$$(1-\sS_{0}(E_1\oplus E_2))=(1-\sS_{0}(E_1))\cdot (1-\sS_{0}(E_1)).$$
And this formula of motivic Milnor fibers holds for every point on $X_1\times X_2$. 
Hence we calculate (let $E:=E_1\oplus E_2$):
\begin{align*}
\ll^{n/2}\odot\sS_{X_1\times X_2}^{\phi}&=\ll^{\dim\ext^0(E,E)/2}\cdot\ll^{-\dim\ext^1(E,E)/2}\odot (\mathds{1}-\sS_{U,f})\cdot  \Upsilon(Q_{X,U,f,i}) \\
&=\ll^{\chi(E_1,E_2)/2}\odot \ll^{\dim\ext^0(E_1,E_1)/2}\ll^{-\dim\ext^1(E_1,E_1)/2}(\mathds{1}-\sS_{U_1,f_1})\cdot \\
&\ll^{\dim\ext^0(E_2,E_2)/2}\ll^{-\dim\ext^1(E_2,E_2)/2}(\mathds{1}-\sS_{U_2,f_2})\cdot \Upsilon(Q_{X_1,U_1,f_1,i_1})\odot \Upsilon(Q_{X_2,U_2,f_2,i_2})\\
&=\ll^{\chi(E_1,E_2)/2}\odot \ll^{\dim\ext^0(E_1,E_1)/2}\cdot \sS_{X_1}^{\phi}\cdot \ll^{\dim\ext^0(E_2,E_2)/2}\cdot \sS_{X_2}^{\phi}\\
&=\ll^{\chi(E_1,E_2)/2}\cdot t_1^{*}\sS_{\sX_1}^{\phi}\cdot  t_2^{*}\sS_{\sX_2}^{\phi}.
\end{align*}
So
\begin{equation}\label{motivic:cycle:split}
t^{*}\sS_{\sX_1\times\sX_2}^{\phi}=\ll^{n/2}\odot\sS_{X_1\times X_2}^{\phi}=\ll^{\chi(\alpha_1,\alpha_2)/2}\cdot t_1^{*}\sS_{\sX_1}^{\phi}\cdot  t_2^{*}\sS_{\sX_2}^{\phi}.
\end{equation} 
Comparing the formulas in (\ref{Iai}), (\ref{Ia1stara2}) and (\ref{motivic:cycle:split}), we have
$$I(\alpha_1\star \alpha_2)=I(\alpha_1)\star I(\alpha_2).$$
So the integration map $I$ is an algebra homomorphism.  $\square$


\subsection*{}

\end{document}